\documentclass[11pt]{article}

\usepackage{tikz}
\usepackage{empheq}
\usepackage[shortlabels,inline]{enumitem}
\setlist[enumerate]{leftmargin=15mm,nosep}

\newcommand*\rectangled[1]{\tikz[baseline=(char.base)]{%
\node[rectangle,fill=blue!20,draw,inner sep=2pt,opacity=0.5,text opacity=1] (char) {#1};}}

\usepackage[doc, wmcom,jvcom]{optional}
\usepackage{xcolor}
\usepackage[colorlinks=true,
linkcolor=refkey,
urlcolor=lblue,
citecolor=red]{hyperref}
\usepackage{float}
\usepackage{soul}
\usepackage{graphicx}
\definecolor{labelkey}{rgb}{0,0.08,0.45}
\definecolor{refkey}{rgb}{0,0.6,0.0}
\definecolor{Brown}{rgb}{0.45,0.0,0.05}
\definecolor{lime}{rgb}{0.00,0.8,0.0}
\definecolor{lblue}{rgb}{0.5,0.5,0.99}
\definecolor{OliveGreen}{rgb}{0,0.6,0}
\usepackage{mathpazo}
\usepackage{subcaption,bm}
\usepackage[color=green!40]{todonotes}
\usepackage[normalem]{ulem}

\colorlet{hlcyan}{cyan!30}

\usepackage{stmaryrd}
\usepackage{amssymb}
\newcommand{\sperp}{{\scriptscriptstyle\perp}}

\makeatletter
\def\namedlabel#1#2{\begingroup
	\def\@currentlabel{#2}%
	\label{#1}\endgroup
}
\makeatother

\newcommand{\seppthree}{\setlength{\itemsep}{-3pt}}

\newcommand*{\tran}{^{\mkern-1.5mu\mathsf{T}}}

\providecommand{\siff}{\Leftrightarrow}
\newcommand{\weakly}{\ensuremath{\:{\rightharpoonup}\:}}

\newcommand{\nnn}{\ensuremath{{n\in{\mathbb N}}}}

\newcommand{\menge}[2]{\big\{{#1}~\big |~{#2}\big\}}

\newcommand{\fenv}[1]%
{\ensuremath{\,\overrightarrow{\operatorname{env}}_{#1}}}
\newcommand{\benv}[1]%
{\ensuremath{\,\overleftarrow{\operatorname{env}}_{#1}}}

\newcommand{\sign}{\ensuremath{\operatorname{sign}}}

\newcommand{\scal}[2]{\left\langle{#1},{#2}  \right\rangle}

\newcommand{\RR}{\ensuremath{\mathbb R}}

\newcommand{\dom}{\ensuremath{\operatorname{dom}}}

\newcommand{\argmin}{\ensuremath{\operatorname{argmin}}}

\newcommand{\prox}{\ensuremath{\operatorname{Prox}}}

\newcommand{\ran}{\ensuremath{{\operatorname{ran}}\,}}
\newcommand{\zer}{\ensuremath{\operatorname{zer}}}

\newcommand{\Id}{\ensuremath{\operatorname{Id}}}
\newcommand{\bDelta}{{\begin{proof} \Delta}}
\newcommand{\bv}{{\begin{proof} v}}
\newcommand{\bT}{{\begin{proof} T}}
\newcommand{\bD}{{\begin{proof} D}}
\newcommand{\bg}{{\begin{proof} g}}

\newcommand{\minimize}[2]{\ensuremath{\underset{\substack{{#1}}}{\mathrm{minimize}}\;\;#2 }}
\newcommand{\gap}{\mathsf{v}}
\newcommand{\rot}{\mathcal{R}}

{\begin{list}{}{%
            \settowidth{\labelwidth}{\textrm{#1~}}%
            \setlength{\leftmargin}{\labelwidth+\labelsep}}}
    {\end{list}}
\usepackage{amsthm}
\usepackage[capitalize,nameinlink]{cleveref}
\crefname{equation}{}{equations}
\crefname{chapter}{Appendix}{chapters}
\crefname{item}{}{items}
\crefname{enumi}{}{}
\newtheorem{theorem}{Theorem}[section]
\newtheorem{lemma}[theorem]{Lemma}

\newtheorem{corollary}[theorem]{Corollary}

\newtheorem{proposition}[theorem]{Proposition}

\newtheorem{definition}[theorem]{Definition}

\newtheorem{algorithm}{Algorithm}

\newtheorem{example}[theorem]{Example}
\newtheorem{ex}[theorem]{Example}
\newtheorem{fact}[theorem]{Fact}
\newtheorem{remark}[theorem]{Remark}

\renewcommand{\labelenumi}{\rm (\roman{enumi})}

\providecommand{\ds}{\displaystyle}

\providecommand{\abs}[1]{\lvert#1\rvert}
\providecommand{\norm}[1]{\lVert#1\rVert}

\providecommand{\RA}{\Rightarrow}

\providecommand{\RR}{\mathbb{R}}
\providecommand{\clint}[1]{\left[#1\right]}
\providecommand{\opint}[1]{\left]#1\right[}
\providecommand{\ocint}[1]{\left]#1\right]}

\providecommand{\ran}{\operatorname{ran}}

\providecommand{\intr}{\operatorname{int}}

\providecommand{\dom}{\operatorname{dom}}

\newcommand{\fix}{\ensuremath{\operatorname{Fix}}}

\providecommand{\gra}{\operatorname{gra}}
\providecommand{\Id}{\operatorname{{ Id}}}

\providecommand{\kk}{{\begin{proof} K}}

\providecommand{\fady}{\varnothing}

\providecommand{\argmin}{\mathrm{arg}\!\min}

\providecommand{\rras}{\rightrightarrows}

\providecommand{\ball}[2]{B(#1;#2)}

\providecommand{\fix}{\operatorname{Fix}}
\providecommand{\ran}{\operatorname{ran}}

\providecommand{\Id}{\operatorname{Id}}

\providecommand{\zer}{\operatorname{zer}}

\providecommand{\fady}{\varnothing}

\newcommand{\cran}{\ensuremath{\overline{\operatorname{ran}}\,}}

\providecommand{\pr}{\operatorname{Prox}}
\providecommand{\Rr}{\operatorname{R}}

\providecommand{\RR}{\mathbb{R}}

\definecolor{myblue}{rgb}{0.9,0.9,0.98}
\newcommand*\mybluebox[1]{%
\colorbox{myblue}{\hspace{1em}#1\hspace{1em}}}

\allowdisplaybreaks 

\usepackage[most]{tcolorbox}
\usepackage{multirow}
\usepackage{cite}

\newtcbox{\mymath}[1][]{%
nobeforeafter, math upper, tcbox raise base,
enhanced, colframe=blue!20!black,
colback=brown!10, boxrule=0.7pt,
#1}

\begin{document}

%

\author{Regina S. Burachik\thanks{Mathematics, UniSA STEM, University of South Australia, Mawson Lakes, S.A. 5095, Australia. E-mails:~regina.burachik@unisa.edu.au, bethany.caldwell@mymail.unisa.edu.au, yalcin.kaya@unisa.edu.au.},\quad Bethany I. Caldwell\footnotemark[1],\quad C. Yal{\c c}{\i}n Kaya\footnotemark[1],\\[2mm] Walaa M. Moursi\thanks{Department of Combinatorics and Optimization, 
University of Waterloo,
Waterloo, Ontario N2L~3G1, Canada.
E-mail: \texttt{walaa.moursi@uwaterloo.ca}.} ~and~
Matthew Saurette\thanks{200 University Ave W,
Waterloo, Ontario N2L~3G1, Canada.
E-mail: \texttt{matthews14@hotmail.ca}.}
}

\title{\textsf{
On the 
Douglas--Rachford
and Peaceman--Rachford algorithms
in the presence of uniform monotonicity\\ and the absence of minimizers 
}
}

\date{May 21, 2024}

\maketitle

\begin{abstract}
The Douglas--Rachford and Peaceman--Rachford 
algorithms have been 
successfully employed to solve convex optimization problems, 
or more generally find zeros of monotone inclusions.
Recently, the behaviour of these methods in the inconsistent case, i.e., 
in the absence of solutions, has triggered significant consideration.
It has been shown that under mild assumptions
the shadow sequence of the Douglas--Rachford 
algorithm converges {\em weakly}  to a generalized 
solution when the underlying operators
are subdifferentials of proper lower semicontinuous 
convex functions. However, no convergence behaviour has been
proved in the case of Peaceman--Rachford algorithm.
In this paper, we prove the \emph{strong} convergence of the shadow sequences 
associated with the Douglas--Rachford 
algorithm and Peaceman--Rachford algorithm in the possibly inconsistent case 
when one of the operators is uniformly monotone
and $3^*$ monotone but not necessarily a subdifferential.
Several examples illustrate and strengthen our conclusion.  We carry out numerical experiments using example instances of optimization problems.
\end{abstract}
{ 


\section{Introduction}	

Throughout, we assume that 	
\begin{empheq}[box=\mybluebox]{equation}
\text{$X$ is
    a  real Hilbert space, 
}
\end{empheq}
with inner product 
$\scal{\cdot}{\cdot}\colon X\times X\to\RR$ 
and induced norm $\|\cdot\|$. 
Let $A\colon X\rras X$. 
The \emph{graph} of $A$
is $\gra A=\menge{(x,x^*)\in X\times X}{x^*\in Ax}$.
Recall $A$ is \emph{monotone} 
if $\{(x,x^*),(y,y^*)\}\subseteq \gra A$
implies that $\scal{x-y}{x^*-y^*}\ge 0$,
and $A$ is \emph{maximally monotone}
if $A$ is monotone and $\gra A$ has no proper extension (in terms of
set inclusion) that preserves the monotonicity of $A$. 
The
\emph{resolvent} of $A$
is
$J_A=(\Id+A)^{-1} $ and the
\emph{reflected resolvent} of $A$
is
$R_A=2J_A-\Id $,
where $\Id\colon X\to X\colon x\mapsto  x$.

Throughout this paper, we assume that 
\begin{empheq}[box=\mybluebox]{equation}
\text{$A\colon X\rras X$ and $B\colon X\rras X$
    are maximally monotone. 
}
\end{empheq}
Consider the monotone inclusion problem: 
\begin{equation}
\tag{P}
\label{P}
\text{Find $x\in X$
    such that $x\in \zer(A+B) := \menge{x\in X}{0\in Ax+Bx}$.}
\end{equation}
Problem \cref{P} is of significant interest in optimization.
Indeed, thanks to Rockafellar's fundamental result
\cite[Theorem~A]{Rock1970}
we know that the \emph{subdifferential 
operator} $\partial f$ of a proper lower semicontinuous convex function
$f\colon X\to \ocint{-\infty,\infty}$
is maximally monotone.
Set $(A,B)=(\partial f, \partial g)$,
where $f$ and $g$ are 
proper lower semicontinuous convex functions
on $X$. Under appropriate constraint qualifications
problem \cref{P} amounts to finding a minimizer 
of $f+g$, equivalently; a zero of $A+B$, provided that one exists. 
For detailed discussions on problem \cref{P}
and the connection to optimization problems, 
we refer the reader to 
\cite{BC2017,Borwein50,Brezis, 	
BurIus,
Comb96,
Rock98,
Simons1,
Simons2,
Zeidler2a,Zeidler2b} and the references therein.

The Douglas--Rachford algorithm
\cite[Algorithm~2]{LM}
is
a successful optimization technique 
to find  a zero of $A+B$ (assuming that one exists)
provided that we have access to the resolvents $J_A$ and $J_B$.
Under the additional assumption that $A$ is {\em uniformly monotone}
(see \cref{def:mon}\cref{def:unimon})
the Peaceman--Rachford algorithm
\cite[Algorithm~1]{LM} can also be used to solve \cref{P}.
Let $\lambda\in \ocint{0,1}$ and set 
\begin{equation}\label{eqn:T}
T=T_\lambda=(1-\lambda) \Id+\lambda R_BR_A.
\end{equation}
Both algorithms operate by iterating the so-called splitting operator
$T_\lambda$, where\footnote{In passing, we point out that 
for $\lambda\in \opint{0,1}$ $T_\lambda $
is a \emph{relaxation} of the Douglas--Rachford operator,
and the corresponding governing and shadow sequences 
enjoy the same convergence behaviour of the Douglas--Rachford
algorithm, see \cite[page~240]{Varga} and also \cite[Theorem~26.11]{BC2017}.}
$\lambda=\tfrac{1}{2}$
in the case of classical Douglas--Rachford 
and $\lambda=1$
in the case of Peaceman--Rachford. Static\footnote{Throughout this paper, we adopt 
the terminology that \emph{static} refers to results or properties 
that involves nonsequential behaviour, whereas 
\emph{dynamic} refers to results or properties that involve
algorithmic behaviour.} connection 
to the set of zeros of $A+B$ is via the identity\footnote{Let $T:X\to X$. The set of \emph{fixed points} of $T$
is $\fix T:=\menge{x\in X}{x=Tx}$.}
(see, e.g., \cite[Proposition~26.1(iii)(a)]{BC2017}) 
\begin{equation}
\label{eq:fix:zer}
\zer(A+B)=J_{A}(\fix R_BR_A)=J_{A}(\fix T_\lambda).
\end{equation}	 
Let $x\in X$. The algorithms produce two sequences:
the governing sequence $(T^nx)_\nnn$ and the shadow sequence 
$(J_AT^nx)_\nnn$.
The dynamic behaviour of the governing sequence
is beautifully ruled and 
clearly determined 
by fundamental results from fixed point theory.
Indeed, in the case of Douglas--Rachford,
because $T$ is \emph{firmly nonexpansive} exactly one of the following
happens:
\begin{enumerate*}
\item 
$\zer(A+B)\neq \fady$; equivalently, $\fix T\neq \fady$.  In this case
$T^n  x\weakly \overline{x}\in \fix T$ (see \cite[Proposition~2]{LM}) and 	
$J_AT^n  x\weakly J_A\overline{x}\in \zer(A+B)$
(see \cite{Svaiter2012}), or	
\item 
$\zer(A+B)=\fady$; equivalently, $\fix T=\fady$. In case (ii) we have that 
$\norm{T^n  x}\to +\infty$ 
(see \cite[Corollary~2.2]{BBR78}). However, and in a pleasant surprise, 
the shadow sequence 
in this case appears to have a will of its own! 
Indeed, 
when $(A,B) =(\partial f, \partial g)$
weak convergence of the shadow sequence has 
been proved in various scenarios under the mild assumption
that the \emph{normal problem} (see \cref{def:gen:Z} below) has a solution.
The normal problem is obtained by perturbing the 
sum $A+B$ in a specific manner using the so-called
\emph{minimal displacement vector} (see \cref{eq:def:v} below).
For a comprehensive reference of these results we refer the reader to 
\cite{130}, \cite{BDM:ORL16}, \cite{101}, \cite{BM:MPA17}
and the references therein. Related results for detection of infeasibility 
appear in, e.g., \cite{BGSB}, \cite{BL21}, \cite{LRY} and \cite{RLY}.
\end{enumerate*}

Nonetheless, no convergence results 
are known for the case of two maximally monotone 
operators that are not necessarily subdifferential operators.
Furthermore, unlike Douglas--Rachford algorithm,
the behaviour of  Peaceman--Rachford algorithm
in the inconsistent case remains a terra incognita, even in optimization settings.

\emph{The goal of this paper is to provide a comprehensive exploration
of  the behaviour of the 
Douglas--Rachford and 
Peaceman--Rachford algorithms 
for two maximally monotone operators 
when one of the underlying operators
is $3^*$ monotone (see \cref{def:mon}\cref{def:strmon}) and uniformly monotone
under additional mild assumptions. 
Our analysis tackles this behavior in the
 possibly inconsistent situation, i.e., when $\zer(A+B)=\fady$,
a situation that is
 hard-to-analyze but common-to-encounter. }

Our main contributions are summarized as follows:

\begin{enumerate}
\renewcommand{\labelenumi}{(C\arabic{enumi})}
\def\cnt{\stepcounter{enumi}C\arabic{enumi}}
\item[\rectangled{\cnt}]  
In \cref{lem:pre:tele} we provide useful identities and 
consequent inequalities which serve as key ingredients in the 
convergence analysis of the algorithms.
\item[\rectangled{\cnt}] 
\label{R2}
Our first main result appears in \cref{thm:DR:convergence}
where we prove the \emph{strong} convergence
of the Douglas--Rachford algorithm 
when one of the operators is 
uniformly monotone with a supercoercive\footnote{Recall that $f$ is {\em supercoercive} when $f(\cdot)/ \|\cdot\|$ is coercive.} modulus
(this is always true when the operator 
is a subdifferential of a uniformly convex function). 
The convergence relies on the mild assumption that a normal solution 
exists.
\item[\rectangled{\cnt}] 
Our second main result appears 
in \cref{thm:PR:convergence} which provides an
analogous conclusion to the result in 
(C2) for the Peaceman--Rachford 
algorithm. Up to the authors' knowledge, this is the first proof 
of the (strong) convergence of the shadow sequence 
of the Peaceman--Rachford algorithm in the inconsistent case.
\item[\rectangled{\cnt}] 
The main optimization results are presented in 
\cref{sec:5}; namely, in \cref{thm:DR:convergence:fg} 
and \cref{thm:PR:convergence:fg}. 
In \cref{prop:example:fg} we provide 
a situation where a unique normal solution exists.
Computationally useful estimates of the gap vector (between the two disjoint domains) are provided in \cref{fact:pazybr}.

\item[\rectangled{\cnt}]
 The above results on the Douglas--Rachford and Peaceman--Rachford algorithms for (possibly inconsistent) optimization problems are numerically illustrated by means of example problems in $d$ variables, for $d=2$ (for which visualization is possible) and for $d \in\{ 70, 100, 1000\}$.  These example problems involve an affine set and a box as the {\em disjoint} domains of the two operators, respectively.  Numerical experiments are carried out to investigate the ``best'' values of the algorithmic parameters, including the relaxation parameter $\lambda$, for various instances of the problems.  Comparisons are also made with (the celebrated) Dykstra's projection algorithm~\cite{BoyleDykstra}.
%
\end{enumerate}

\subsection*{Organization and notation}
This paper is organized as follows: 
\cref{sec:2} contains a collection of auxiliary results
and examples of resolvents that are not necessarily 
proximal mappings.
In \cref{section:3} we provide key results concerning 
the generalized fixed point set and connection to the normal 
(generalized) solutions.
Our main results appear in \cref{sec:4} 
and the optimization counterpart of these results 
appear in \cref{sec:5}.  Numerical experiments are presented in \cref{sec:num_exp}.  Finally, \cref{sec:conc} provides concluding remarks and future directions of research.

Our notation is standard and follows largely, e.g., 
\cite{Rock70} and \cite{BC2017}.





\section{Auxiliary results}
\label{sec:2}	
\subsection{Nonexpansive mappings and their minimal displacement vector}

\begin{definition}
    Let $C$ be a nonempty closed convex subset of $X$ and fix $x\in X$. Then $p\in C$ is the {\em projection} of $x$ onto $C$ if $\|x-p\|= \inf_{y\in C} \|x-y\| =: d_C(x)$. We denote $p$ as $P_C(x)$.
\end{definition} 
\begin{lemma}
\label{lem:proj:scal}
Let $\alpha\in \RR $ and let $C$ 
be a nonempty closed convex subset of $X$.
Then $P_{\alpha C}(0)=\alpha P_{C}(0)$.
\end{lemma}	
\begin{proof}
Indeed, let $p\in C$.
Then, by \cite[Theorem 3.16]{BC2017}, we have that
$p=P_{C}(0)\siff (\forall c\in C) $
$\scal{0-p}{c-p}\le 0\siff (\forall c\in C) \scal{0-\alpha p}{\alpha c- \alpha p}\le 0
\siff \alpha p= P_{\alpha C}(0)$.
\end{proof}	

\begin{definition}\label{def:nonexp+}
    Let $T\colon X\to X$ and let $(x,y)\in X\times X$ and let $\lambda\in \left]0,1\right[$.
Recall that  $T$ 
is \emph{nonexpansive} if 
$\norm{Tx-Ty}\le \norm{x-y}$
and $T$ 
is \emph{$\lambda$-averaged} if 
$\norm{Tx-Ty}^2+\left(\frac{1-\lambda}{\lambda}\right)\norm{(\Id-T)x-(\Id-T)y}^2\le \norm{x-y}^2$.
\end{definition}

The next result guarantees the convexity of 
the range of the displacement mapping
associated with nonexpansive mappings.
\begin{lemma}
\label{lem:rancon}
Let $\lambda\ge 0$ and let $N\colon X\to X$
be nonexpansive. 
Then $\cran (\lambda(\Id-N))$ is convex.
\end{lemma}	
\begin{proof}
This is a direct consequence of  \cite[Example~20.29~and~Corollary~21.14]{BC2017}.
\end{proof}	

\begin{lemma}
\label{lem:fix:shift}
Let $T\colon X\to X$ and let $w\in X$.
Then $\fix T(\cdot +w)=-w+\fix (w+T)$.
\end{lemma}	
\begin{proof}
Indeed, let $x\in X$.
Then 
$x\in \fix T(\cdot +w)
\siff x =T(x +w)
\siff -w+x+w=T(x+w)
\siff x+w\in \fix (w+T)
\siff x\in -w+\fix (w+T)$.
\end{proof}

\begin{lemma}
\label{lem:fix:p}
Let $T\colon X\to X$ be nonexpansive,
let $\gap=P_{{\cran}(\Id-T)}0$
and 
suppose that $x\in \fix (\gap+T)$. 
Let $\nnn$. Then $T^n x=x-n\gap$.
\end{lemma}	
\begin{proof}
We proceed by induction on $\nnn$.
The base case at $n=0$ is clear.
Now suppose for some  $\nnn$
we have 
$T^n x=x-n\gap$.
The definition of $\gap$ and the nonexpansiveness of $T$
imply
$\norm{\gap}\le \norm{(\Id-T)T^n x}=\norm{T^n x-T^{n+1}x}\le \norm{x-Tx}=\norm{\gap}$. The first inequality follows from the fact that $v$ is the element in ${\cran}(\Id-T)$ with minimum norm, and the fact that $(\Id-T)T^n x\in {\ran}(\Id-T)\subseteq {\cran}(\Id-T)$. The last equality follows from the definition of $x$. Therefore,
$\norm{T^n x-T^{n+1}x}=\norm{\gap}$. By the uniqueness of $v$ we conclude that $T^n x-T^{n+1}x=\gap$.
Now use the inductive hypothesis to complete the proof.	 
\end{proof}	

Let 
$\lambda\in \left]0,1\right[$
and suppose that
$T\colon X\to X$ is  $\lambda$-averaged.
Let $x\in X$.
Recall that (see, e.g., \cite[Proposition~1.2]{BR77})
\begin{equation}
\label{e:v:limit}
T^n x- T^{n+1} x\to P_{\cran(\Id-T)}0.
\end{equation}			

\subsection{Further notions of monotonicity}
\begin{definition}\label{def:mon}
Let $C\colon X\rras X$ be monotone. 
Then
\begin{enumerate}
\item \label{def:unimon}
			$C $ is \emph{ uniformly monotone} with a modulus 
			$\phi\colon \RR_+\to\left[0,+\infty\right] $ if 
			$\phi$ is increasing, vanishes only at $0$, and 
			
   \begin{equation}
   \{(x,x^*),(y,y^*)\}\subseteq \gra C
\RA \scal{x-y}{x^*-y^*}\ge \phi(\norm{x-y}).
\end{equation}
\item \label{def:strmon}
$C $ is \emph{strongly monotone} with a constant $\beta>0$
if $C-\beta \Id$ is monotone, i.e., 
   \begin{equation}
   \{(x,x^*),(y,y^*)\}\subseteq \gra C
\RA \scal{x-y}{x^*-y^*}\ge \beta \norm{x-y}^2.
\end{equation}
\item 
$C $ is \emph{$3^*$ monotone} if $(\forall (y,z^*)\in \dom C\times \ran C)$ 
\begin{equation}
 \inf_{(x,x^*)\in\gra C} \scal{x-y}{x^*-z^*}>-\infty.
\end{equation}
\item $C$ is
\emph{cyclically monotone} if, for every $n\ge 2$,
for every $(x_1,\ldots,x_{n+1})\in X^{n+1}$
and every $(x_1^*,\ldots,x_{n}^*)\in X^{n}$ 
\begin{align}
\label{eq:def:sn}
\left.\begin{array}{r@{\mskip\thickmuskip}l}
    (x_1,x_1^*)\in \gra C \\
    { \vdots\;\;\;\;\;\;\;\;\;}\\
    (x_n,x_n^*)\in \gra C \\
    x_{n+1}=x_1
\end{array} \right\}
\quad \implies \quad
\sum_{i=1}^{n}\scal{x_{i+1}-x_i}{x_i^*}\le  0.
\end{align}
\end{enumerate}
\end{definition}
The importance of uniform (respectively strong) monotonicity is,
in fact, motivated by its close connection to 
the notions of uniform (respectively strong)
convexity as we see below.
\begin{fact}
\label{f:func-sd}
Let $f\colon X\to \left]-\infty,+\infty\right]$.
\begin{enumerate}
    \item 
    \label{f:func-sd:i}
    Suppose that $f$
				is uniformly convex with modulus 
				$\phi$.
				Then $\partial f$ is uniformly monotone with modulus $2\phi$. 
    \item 
    \label{f:func-sd:ii}
    Suppose that $f$ is $\beta$-strongly convex for some $\beta>0$.
    Then $\partial f$ is $\beta$-strongly monotone.
\end{enumerate}
				
			\end{fact}	
			\begin{proof}
   \cref{f:func-sd:i}:
				See \cite[Theorem~3.5.10]{Za02} and also \cite[Example~22.4(iii)]{BC2017}. \\
       \cref{f:func-sd:ii}:
				See \cite[Example~22.4(iv)]{BC2017}.
			\end{proof}	

\subsection{Examples of resolvents that are not necessarily proximal mappings}

We now present a collection of resolvents of maximally monotone operators that are not necessarily subdifferentials. These results are interesting on their own, since the computation of the resolvent is not straightforward in general.

\begin{proposition}
\label{ex:general:Id:S:res}
Let $(\alpha,\beta)\in \RR_+\times \RR$.
Suppose that $S\colon X\to X$ is continuous, linear, and single-valued such
that  $S$ and $-S$ are monotone and $S^2=-\gamma \Id$
where $\gamma \ge 0$.
Suppose that $A=\alpha \Id+\beta S$.
Then 
\begin{equation}
    \label{eq:skew:resol:form}
    J_A	=\tfrac{1}{(1+\alpha)^2+\beta^2\gamma }((1+\alpha) \Id-\beta S),
\end{equation}	
and
\begin{equation}
    \label{eq:skew:rresol:form}
    R_A	=\tfrac{1-\alpha^2-\beta^2\gamma}{(1+\alpha)^2+\beta^2\gamma } \Id-\tfrac{2\beta}{(1+\alpha)^2+\beta^2\gamma} S.
\end{equation}	
\end{proposition}	
\begin{proof}
Indeed, set $T=\tfrac{1}{(1+\alpha)^2+\beta^2\gamma }((1+\alpha) \Id-\beta S)$.
Then  
\begin{subequations}
    \begin{align}
        (\Id+A)T&=\tfrac{1}{(1+\alpha)^2+\beta^2\gamma }((1+\alpha) \Id+\beta S)((1+\alpha) \Id-\beta S)
        \\
        &=	\tfrac{1}{(1+\alpha)^2+\beta^2\gamma }((1+\alpha)^2\Id +(1+\alpha)\beta S-\beta^2S^2-(1+\alpha)\beta S)
        \\
        &=\Id.
    \end{align}
\end{subequations}	
This proves \cref {eq:skew:resol:form}. The formula 
in \cref {eq:skew:rresol:form} is a direct consequence of 
\cref{eq:skew:resol:form}. 
\end{proof}

Let $L\colon X\to X$ be continuous and linear.
Recall that the \emph{adjoint}
of $L$  is the unique linear operator $L^*\colon X\to X$
that satisfies $(\forall (x,y)\in X\times X)$
$\scal{x}{L^*y}=\scal{Lx}{y}$.

In the following, we use
$\ball{z}{\rho}$ to denote the {\em closed ball} in $X$ with centre $z$ and radius $\rho$. Namely,
\[
\ball{z}{\rho}:=\{ y\in X\::\: \|z-y\|\le \rho\}.
\]
Given a closed set $C\subseteq X$, the {\em normal cone operator} associated with $C$ is denoted by $N_C$ and defined as
    \[
    N_C(x):=\left\{ 
    \begin{array}{lc}
        \{w\in X\::\: \langle w,y-x \rangle \le 0,\, \forall y\in C\}\,, &  \text{ if }x\in C, \\
         \varnothing\,, &  \text{ if }x\not\in C.
    \end{array}\right.
    \]

\begin{proposition}
\label{ex:Id:S}
Let $(\alpha,\beta)\in  \RR_+\times \RR$.
Suppose that $S\colon X\to X$ is continuous, linear, and single-valued such
that  $S$ is monotone, $S^*=-S$ 
and $S^2=-\gamma \Id$
where $\gamma \ge 0$.
Suppose that $L=\alpha \Id+\beta S$
and set $A=L+N_{\ball{0}{1}}$.
Then $A$ is maximally monotone and 
\begin{equation}
    \label{eq:nice:res}
    J_A	x=
    \begin{cases}
        \tfrac{1}{(1+\alpha)^2+\beta^2\gamma }((1+\alpha) \Id-\beta S)x,
        &\norm{x}^2\le (1+\alpha)^2+\beta^2\gamma;
        \\
        \tfrac{1}{\norm{x}^2}\Big(\sqrt{\norm{x}^2- \beta^2\gamma}\Id-\beta  S\Big)x,
        &\text{otherwise}.
    \end{cases}
\end{equation}	
\end{proposition}	

\begin{proof}
The maximal monotonicity of 
$A$ follows from, e.g., \cite[Corollary~25.5(i)]{BC2017}.	
We now turn to \cref{eq:nice:res}.
Indeed, let $u\in X$ and 
observe that 	
$u=J_A x \siff x\in (1+\alpha )u+\beta Su+N_{\ball{0}{1}}u
\siff (\exists r\ge 0)$ 
\begin{equation}
    \label{e:key:connec}
    x= (1+\alpha )u+\beta Su+ru,
\end{equation}	
where we used \cite[Example 6.39]{BC2017}. We proceed by cases.
\textsc{Case~1:}
$\norm{u}<1$. In this case $r=0$
and \cref{e:key:connec} yields $x= (1+\alpha )u+\beta Su$.
By \eqref{e:key:connec} with $r=0$ we deduce that $\norm{x}^2=((1+\alpha)^2+\gamma \beta^2)\norm{u}^2$
and also that $u=J_Lx$. 
Now combine with \cref{ex:general:Id:S:res}
applied with $A$ replaced by $L$.
\textsc{Case~2:}
$\norm{u}=1$.
It follows from 
\cref{e:key:connec} that $r=\scal{x}{u}-(1+\alpha)$.
Therefore, \cref{e:key:connec}  becomes
$x=\scal{x}{u}u+\beta Su$.
Hence $\norm{x}^2=\norm{\scal{x}{u}u+\beta Su}^2=\scal{x}{u}^2+\beta^2\gamma$.
We therefore learn that 	$\scal{x}{u}^2=\norm{x}^2-\gamma \beta^2$. Because $J_A$ is (maximally) monotone by, e.g.,
\cite[Corollary~23.11(i)]{BC2017}
and $\{(0,0), (x,u)\}\subseteq \gra J_A$ 
we learn that $\scal{x}{u}\ge0$.
Therefore, we conclude that
$\scal{x}{u}=\sqrt{\norm{x}^2-\beta^2\gamma}$.
Altogether we rewrite \cref{e:key:connec}  as 
\begin{equation}
    \label{e:key:connec:i} 
    x=\sqrt{\norm{x}^2-\beta^2\gamma}u+\beta Su.
\end{equation}	
It is straightforward to verify that 
$u=\tfrac{1}{\norm{x}^2}(\sqrt{\norm{x}^2- \beta^2\gamma}\Id-\beta  S)x$ satisfies \cref{e:key:connec:i}.
The proof is complete.
\end{proof}

\begin{example}
\label{ex:Id:S:concrete}
Suppose that $X=\RR^2$,
let $\theta\in \left[0,\tfrac{\pi}{2}\right[$,
set 
\begin{equation}
    \rot=\rot_\theta=	\begin{pmatrix}
        \cos(\theta) &-\sin (\theta )
        \\
        \sin (\theta) &	\cos(\theta)
    \end{pmatrix},
\end{equation}
and set $A=\rot+N_{\ball{0}{1}}$.
Then $A$ is maximally monotone and 
strongly monotone and 
\begin{equation}
    J_A	x=
    \begin{cases}
        \tfrac{1}{2(1+\cos(\theta)) }(\Id+\rot_{-\theta})x,
        &\norm{x}^2\le 2(1+\cos\theta);
        \\
        \tfrac{1}{\norm{x}^2}\Big(\Big(\sqrt{\norm{x}^2
            - \sin^2(\theta )}-\cos(\theta)\Big)\Id+\rot_{-\theta}\Big)x,
        &\text{otherwise}.
    \end{cases}
\end{equation}	
\end{example}	
\begin{proof}
    Set $S=\begin{psmallmatrix}
    0&-1\\
    1&0
\end{psmallmatrix}$ and observe that 
$S$ is maximally monotone,
$S^{\mathsf{T}}=-S$,
$S^2=-\Id$, 
and $\rot=\cos( \theta) \Id+\sin (\theta )S$.
It is straightforward to verify that, likewise $\rot$, 
$A$ is $\cos (\theta) $-strongly monotone.
The conclusion now
follows from applying \cref{ex:Id:S} with $(\alpha,\beta,\gamma)$
replaced by $(\cos(\theta),\sin (\theta),1)$. 
\end{proof}	

\begin{example}
\label{ex:K:R}
Suppose that $X=\RR^2$, let $(u,u^\perp) \in \RR^2\times \RR^2$ 
such that 
$\norm{u}=\norm{u^\perp}=1$, $\scal{u}{u^{\sperp}}=0$,
and set $K=\menge{x\in X}{\scal{x}{u}\le 0}$.
Let 
$L
=\begin{pmatrix}
    \alpha&\beta\\
    \gamma&\delta
\end{pmatrix}	\in \RR^{2\times 2}$
be monotone\footnote{Recall that $L$
    is monotone if and only if $\alpha\ge 0$,
    $\delta\ge 0$ and $4\alpha\delta\ge (\beta+\gamma)^2$ 
    (see \cite[Lemma~6.1]{Lukenspaper}).
}
and set $A= L+N_K$.
Set $\kappa=1+\scal{u^\sperp}{Lu^\sperp}\geq1$ by the monotonicity of $L$.
Let $x\in \RR^2$.
Then we have
\begin{equation}
    \label{eq:res:L}
    J_L 
    =\tfrac{1}{(1+\alpha)(1+\gamma)-\beta\delta}
    \begin{pmatrix}
        1+\delta&-\beta
        \\
        -\gamma&1+\alpha	
    \end{pmatrix}	,
\end{equation}
and  
\begin{equation}
    \label{eq:res:A}
    J_A x
    =
    \begin{cases}
        J_Lx, &\scal{J_Lx}{u}<0;
        \\
        \tfrac{1}{\kappa} \scal{x}{u^\sperp}	u^\sperp,
        &\text{otherwise}.
    \end{cases}	
\end{equation}
\end{example}

\begin{proof}
The formula in \cref{eq:res:L} is clear.
We verify  \cref{eq:res:A}.
Let $(x,z)\in X\times X$.
Then $z=J_Ax\siff x\in z+Lz+N_Kz$. Observe that this implies that $z\in K$ and therefore we consider two cases.

\textsc{Case~1:} $\scal{z}{u}<0$.
This implies that $N_K z=\{0\}$
and, therefore, $x=(\Id+L)z$.
Equivalently, $z=J_Lx$.

\textsc{Case~2:} $\scal{z}{u}=0$.	
Observe that in this case
$N_Kz=\RR_+ u$. 
That is, $(\exists r\ge 0)$
such that $x=z+Lz+ru$.
Therefore, $r=\scal{x-Lz}{u}$.
We claim that $z=\tfrac{1}{\kappa}\scal{x}{u^\sperp}u^\sperp$
solves the equation
\begin{equation}
    \label{ex:verify:proj}
    x=z+Lz+\scal{x-Lz}{u}u.
\end{equation}	
Indeed, substituting for $z=\tfrac{1}{\kappa}\scal{x}{u^\sperp}u^\sperp$
in the right hand side of \cref{ex:verify:proj} yields
\begin{subequations}
    \begin{align}
        z+Lz+\scal{x-Lz}{u}u
        &=z
        +\scal{Lz}{u}u
        +\scal{Lz}{u^\sperp}u^\sperp
        +\scal{x}{u}u
        -\scal{Lz}{u}u
        \\
        &=z
        +\scal{Lz}{u^\sperp}u^\sperp
        +\scal{x}{u}u
        \\
        &=\tfrac{1}{\kappa}\scal{x}{u^\sperp}u^\sperp
        +\tfrac{1}{\kappa}\scal{Lu^\sperp}{u^\sperp}\scal{x}{u^\sperp}u^\sperp
        +\scal{x}{u}u
        \\
        &=\tfrac{1}{\kappa}\big(1+\scal{Lu^\sperp}{u^\sperp}\big)\scal{x}{u^\sperp}u^\sperp
        +\scal{x}{u}u
        \\
        &=\scal{x}{u^\sperp}u^\sperp
        +\scal{x}{u}u=x.
    \end{align}	
\end{subequations}
Hence, 
$z=\tfrac{1}{\kappa}\scal{x}{u^\sperp}u^\sperp$
solves   \cref{ex:verify:proj} as claimed.
The proof is complete. 
\end{proof}

\begin{example}
\label{ex:3*:not:sub}
Suppose that $X=\RR^2$,
let $\theta\in \left]0,\tfrac{\pi}{2}\right[$,
set 
\begin{equation}
    \rot=\rot_\theta=	
    \begin{pmatrix}
        \cos(\theta )&-\sin (\theta )
        \\
        \sin (\theta )&	\cos(\theta)
    \end{pmatrix},
\end{equation}
and let $U$ be nonempty closed convex subset in $\RR^2$.
Set 
\begin{equation}
    A=\rot+N_U.
\end{equation}
Then the following hold:
\begin{enumerate}
    \item
    \label{ex:3*:not:sub:i}
    $A$ is maximally monotone and strongly monotone.
    \item
    \label{ex:3*:not:sub:ii}
    $A$ is $3^*$ monotone.
    \item
    \label{ex:3*:not:sub:iii}
    $A$ is \emph{not}  cyclically monotone.
    Hence, $A$ is not a subdifferential operator.
\end{enumerate}	

\end{example}	

\begin{proof}
Set $S=\begin{psmallmatrix}
    0&-1\\
    1&0
\end{psmallmatrix}$ and observe that 
 $\rot=\cos (\theta) \Id+\sin (\theta) S$.

\cref{ex:3*:not:sub:i}:	
Clearly $\rot$ is maximally monotone and $\cos \theta$-strongly monotone and 
$N_U$ is maximally monotone.
The  maximal monotonicity of $A$ follows from, e.g., \cite[Corollary~25.5(i)]{BC2017}.
The  strong monotonicity of $A$  is an immediate consequence of 
the  strong monotonicity of $\rot$.

\cref{ex:3*:not:sub:ii}:	Combine \cref{ex:3*:not:sub:i} and \cite[Example~25.15(iv)]{BC2017}.

\cref{ex:3*:not:sub:iii}:
It follows from \cite[Example~4.6]{BBBR07} that 
$A$ is not  cyclically monotone.
Therefore by 
\cite[Theorem~B]{Rock1970} $A$ is not a subdifferential
operator.
%
%
%
\end{proof}

\begin{example}
\label{ex:conc:3*:not:sub}
Suppose that $X=\RR^2$, let 
\begin{equation}
    \rot=	\begin{pmatrix}
        \frac{1}{2} &-\frac{\sqrt{3}}{2}
        \\
        \frac{\sqrt{3}}{2} &	\frac{1}{2}
    \end{pmatrix},
\end{equation}
and let $K= \RR_{-}\times \RR$.

\begin{equation}
    A=\rot+N_K.
\end{equation}
Then the following hold:
\begin{enumerate}
    \item
    \label{ex:conc:3*:not:sub:i}
    $A$ is maximally monotone and strongly monotone.
    \item
    \label{ex:conc:3*:not:sub:ii}
    $A$ is $3^*$ monotone.
    \item
    \label{ex:conc:3*:not:sub:iii}
    $A$ is \emph{not}  cyclically monotone.
    Hence, $A$ is not a subdifferential operator.
    \item
    \label{ex:conc:3*:not:sub:iv}
    We have 
    \begin{equation}
        J_A\colon (\xi_1,\xi_2)\mapsto
        \begin{cases}
            \tfrac{1}{3} (\Id-R)(\xi_1,\xi_2), &\xi_1+\sqrt{3}\xi_2<0;
            \\
            \tfrac{2}{3} (0,\xi_2),
            &\text{otherwise}.
        \end{cases}	
    \end{equation}
\end{enumerate}		
\end{example}	

\begin{proof}
\cref{ex:conc:3*:not:sub:i}--\cref{ex:conc:3*:not:sub:iii}:
Apply \cref{ex:3*:not:sub}\cref{ex:3*:not:sub:i}--\cref{ex:3*:not:sub:iii}
with $(\theta,U)$ replaced by $(\tfrac{\pi}{3},K)$.
\cref{ex:conc:3*:not:sub:iv}:
Write  $K=\menge{(\xi_1,\xi_2)\in \RR^2}{\scal{(\xi_1,\xi_2)}{(1,0)}\le 0}$.
Now apply \cref{ex:K:R} with $(u,u^\sperp)$
replaced by $((1,0),(0,1))$
and 
$(\alpha,\beta,\gamma,\delta)$
replaced by 
$\Big(\tfrac{1}{2},-\tfrac{\sqrt{3}}{2},\tfrac{\sqrt{3}}{2},\tfrac{1}{2}\Big)$.
\end{proof}

\section{Generalized solutions and generalized fixed points}
\label{section:3}
Throughout the remainder of this paper, we set
\begin{empheq}[box=\mybluebox]{equation}
\label{def:T(AB)}
T=T_{(A,B)}=\tfrac{1}{2}(\Id+R_BR_A).
\end{empheq}
The well-defined 
\emph{minimal displacement vector }associated with $T$
(see \cref{lem:rancon})
is
\begin{empheq}[box=\mybluebox]{equation}
\label{eq:def:v}
\gap=P_{\cran(\Id-T)}(0),
\end{empheq}
and the \emph{generalized solution set} (this is also known as 
the \emph{set of normal solutions} (see \cite{Sicon})) is
\begin{empheq}[box=\mybluebox]{equation}
\label{def:gen:Z}
Z=\zer(-\gap+A+B(\cdot-\gap))=\menge{x\in X}{0\in -\gap+Ax+B(x-\gap)}.
\end{empheq}

We recall that
(see \cite[Theorem~2.5]{WMM2024})

\begin{equation}
\label{eq:A:ran:T}
{\cran}(\Id-T)=\overline{(\dom A-\dom B)\cap
    ( \ran{A}+\ran B)}.
\end{equation}
The following remark gives situations under which assumption \eqref{eq:A:ran:T} holds. We recall that case (i) below automatically holds when $(A,B)=(\partial f,\partial g)$ for $f,g$ convex, proper and lsc functions.
\begin{remark}
Assumption \cref{eq:A:ran:T} holds if 
one of the following holds (see \cite[Theorem~5.2]{MOR}):
\begin{enumerate*}
    \item
    $A$ and $B$ are $3^*$ monotone.
    \item
    $(\exists C\in \{A,B\})$
    $\dom C=X$ and C is $3^*$ monotone.
    \item
    $(\exists C\in \{A,B\})$
    $\ran C=X$ and C is $3^*$ monotone.
\end{enumerate*}		
\end{remark}	
\begin{lemma}
\label{lem:zeros}
Let $Z$ be as in \eqref{def:gen:Z}. We have $Z=J_A(\fix (\gap+T))$.	
\end{lemma}	
\begin{proof}
Indeed, it follows from \cite[Proposition~2.24]{Sicon}
that 
$T_{(-\gap+A,B(\cdot-\gap))}=T(\cdot+\gap)$ 
Combining with \cref{eq:fix:zer}, applied with
$(A,B) $
replaced by $(-\gap+A,B(\cdot-\gap))$,
and \cite[Proposition~23.17(ii)]{BC2017},
we learn that 	$Z=J_{-\gap +A}(\fix T(\cdot +\gap))
=J_A(\gap+\fix T(\cdot +\gap))$.	
The claim now follows from combining the last equality with \cref{lem:fix:shift} 
applied with $w$ replaced by $\gap$.
\end{proof}

\begin{lemma}
\label{lem:static:fix:zer}
Let $\lambda\in \left]0,1\right]$, set 
\begin{equation}\label{def:Tlambda}
    T_\lambda = (1-\lambda) \Id +\lambda R_BR_A,
\end{equation}	
and set $\gap_\lambda=
P_{\cran({\Id-T_\lambda})}(0)$.
Then $\gap_\lambda$ is well defined. Moreover
the following hold:
\begin{enumerate}
    \item
    \label{lem:static:fix:zer:0}
    ${\Id-T_\lambda}=\lambda(\Id-R_BR_A)=2\lambda(J_A-J_BR_A)$.
    \item
    \label{lem:static:fix:zer:i}
    $\gap_\lambda=2\lambda \gap$.
    \item
    \label{lem:static:fix:zer:ii}
    $\fix (\gap_\lambda+T_\lambda)
    =\fix(\gap+T)$.
\end{enumerate}

\end{lemma}
\begin{proof}
The claim that $\gap_\lambda $ is well defined 
follows from applying \cref{lem:rancon} with  $N$ replaced by $R_BR_A$.
\cref{lem:static:fix:zer:0}:
This fact clearly follows from \eqref{def:Tlambda} and the definitions.
\cref{lem:static:fix:zer:i}:
Recalling \cref{def:T(AB)},
observe that 
$T=T_{1/2}$ and hence 
$\cran(\Id-T_\lambda)=\cran(\Id-(1-\lambda)\Id-\lambda R_B R_A) = \cran(-\lambda(-\Id+R_B R_A)) = \cran(\lambda(2\Id-\Id-R_B R_A)) = \cran(2\lambda(\Id-\frac{1}{2}\Id-\frac{1}{2}R_B R_A)) =2\lambda \cran (\Id-T)$.
Now combine with \cref{lem:proj:scal} applied with 
$(\alpha, C) $ replaced by $(2\lambda,\cran(\Id-T)  )$.
\cref{lem:static:fix:zer:ii}:
Indeed, let $x\in X$. Then $x\in \fix (\gap_\lambda+T_\lambda)
\siff x=\gap_\lambda+(1-\lambda )x+\lambda R_BR_Ax
\siff \lambda x=\gap_\lambda+\lambda R_BR_Ax
\siff x=\tfrac{\gap_\lambda}{2\lambda } +\frac{1}{2}(x+R_BR_Ax)$.
Now combine with \cref{lem:static:fix:zer:i}.
\end{proof}	


\begin{proposition}
\label{fact:fix:p}
Let $\lambda\in \left]0,1\right]$ and set
$T_\lambda =(1-\lambda) \Id +\lambda R_BR_A
$.
Let	$x\in X$ and let $\nnn$. 
Suppose that 
$y\in \fix(\gap +T)$. Then the following hold:
\begin{enumerate}
    \item
    \label{fact:fix:p:i}
    $T_\lambda^n y=y-2\lambda n \gap \in \fix(\gap+T)$.
%
    \item
    \label{fact:fix:p:ii}
    Suppose that $Z=\{\overline{x}\}$. Then
    $J_AT_\lambda^n y=J_A(y-2\lambda n\gap)=\overline{x}$.
    \end{enumerate}
    Suppose that $Z=\{\overline{x}\}$ and $\lambda\in \left]0,1\right[$.
    Then we additionally have: 
    \begin{enumerate}
        \setcounter{enumi}{2}
    \item
    \label{fact:fix:p:0}
    $J_AT_\lambda^n x-J_BR_AT_\lambda^nx
    =\tfrac{1}{2\lambda}(\Id-T_\lambda)T_\lambda^n x
    =\tfrac{1}{2\lambda}(T_\lambda^{n}x-T_\lambda^{n+1}x)\to \tfrac{1}{2\lambda}\gap_\lambda=\gap
    $.
    \item
    \label{fact:fix:p:iii}
    $J_BR_AT_\lambda^n y
    =J_BR_A(y-2\lambda n\gap)=\overline{x}-\gap$.
\end{enumerate}
\end{proposition}	

\begin{proof}
    \cref{fact:fix:p:i}: 
    Combine \cref{lem:static:fix:zer}\cref{lem:static:fix:zer:i} 
and \cref{lem:fix:p} with $T$ replaced by $T_\lambda$.
\cref{fact:fix:p:ii}: 	
The first identity is a direct consequence of \cref{fact:fix:p:i}.
Now on the one hand, we have $Z=J_A(\fix(\gap+T))=\{\overline{x}\}$.
On the other hand, 
by \cref{fact:fix:p:i} we have $J_A(y-2\lambda nv)\in J_A(\fix(\gap+T))$.
Altogether, the conclusion follows.
\cref{fact:fix:p:0}:	
It follows from \cref{lem:static:fix:zer}\cref{lem:static:fix:zer:0}
that 
$J_AT_\lambda^n x-J_BR_AT_\lambda^nx
=\tfrac{1}{2\lambda}(\Id-T_\lambda)T_\lambda^n x$.
Now combine with \cref{e:v:limit} 
applied with $T$ replaced by $T_\lambda$
in view of \cref{lem:static:fix:zer}\cref{lem:static:fix:zer:i}.

\cref{fact:fix:p:iii}:
Combine  \cref{fact:fix:p:ii} and \cref{fact:fix:p:0}
applied with $x$ replaced by $y$.
\end{proof}	

We now recall the following 
key result by Minty which is of central importance in our proofs.
\begin{fact}[{\bf Minty's Theorem}]
\label{thm:minty}
Let $C\colon X\rras X$ be monotone. Then
\begin{equation}
    \label{eq:Minty}
    \gra C=\menge{(J_C x, J_{C^{-1}}x)}{x\in \ran (\Id+C)}.
\end{equation}
Moreover,
\begin{equation}
    \label{eq:Minty:2}
    \text{$C$ is maximally monotone $\siff$ $\ran (\Id+C)=X$.}
\end{equation}
\end{fact}	
\begin{proof}	
See \cite{Minty}.
\end{proof}

\begin{lemma}
\label{lem:pre:tele}
Let $\lambda\in \left]0,1\right]$ and set
\begin{equation}
    T_\lambda =(1-\lambda) \Id +\lambda R_BR_A.
\end{equation}	
Let $(x,y)\in X\times X$.
Then
\begin{align}
    \label{e:general:ineq}
    &\qquad\lambda \norm{x-y}^2 -\lambda \norm{T_\lambda   x-T_\lambda  y}^2-(1-\lambda) \norm{(\Id-T_\lambda  )x-(\Id-T_\lambda  )y}^2
    \nonumber
    \\
    &=4\lambda^2\scal{J_Ax-J_Ay}{J_{A^{-1}}x-J_{A^{-1}}y}
    +4\lambda^2\scal{J_BR_Ax-J_BR_Ay}{J_{B^{-1}}R_Ax-J_{B^{-1}}R_Ay}.
\end{align}	
Consequently we have:
\begin{enumerate}
    \item
    \label{lem:pre:tele:i}
    $ \norm{x-y}^2 - \norm{T_\lambda  x-T_\lambda  y}^2\ge 4\lambda\scal{J_Ax-J_Ay}{J_{A^{-1}}x-J_{A^{-1}}y}$.
    \item
    \label{lem:pre:tele:ii}
    $ \norm{x-y}^2 - \norm{T_\lambda  x-T_\lambda  y}^2\ge 4\lambda \scal{J_BR_Ax-J_BR_Ay}{J_{B^{-1}}R_Ax-J_{B^{-1}}R_Ay}$.
\end{enumerate}	

\end{lemma}

\begin{proof}
Indeed, observe that
\begin{equation}
    \label{eq:Tlamda}
    T_\lambda  
    =(1-2\lambda)\Id +\lambda (\Id+R_BR_A)
\end{equation}
and 
\begin{equation}
    \label{eq:Tlamda:disp}
    \Id-T_\lambda  
    =\lambda(\Id-R_BR_A).
\end{equation}
In view of \cref{eq:Tlamda} and \cref{eq:Tlamda:disp} 
we have 
\begin{subequations}
    \begin{align}
        &\quad
        \lambda \norm{x-y}^2 -\lambda \norm{T_\lambda x-T_\lambda y}^2-(1-\lambda) \norm{(\Id-T_\lambda)x-(\Id-T_\lambda)y}^2
        \\
        &=
        \scal{\lambda((x-y)+(T_\lambda  x-T_\lambda  y))}{(x-y)-(T_\lambda  x-T_\lambda  y)}-(1-\lambda) \norm{(\Id-T_\lambda  )x-(\Id-T_\lambda  )y}^2
        \\
        &=
        \scal{\lambda((x-y)+(T_\lambda  x-T_\lambda  y))-(1-\lambda)((x-y)-(T_\lambda  x-T_\lambda  y))}{(x-y)-(T_\lambda  x-T_\lambda  y)}
        \\
        &=\scal{T_\lambda  x-T_\lambda  y-(1-2\lambda)(x-y)}{(x-y)-(T_\lambda  x-T_\lambda  y)}
        \\
        &  = \lambda^2\scal{(x-y)+(R_BR_Ax-R_BR_Ay)}{(x-y)-(R_BR_Ax-R_BR_Ay)}
        \\
        &=\lambda^2\big(\norm{x-y}^2-\norm{R_BR_Ax-R_BR_Ay}^2\big)
        \\
        &=\lambda^2\big(\norm{x-y}^2-\norm{R_Ax-R_Ay}^2
        +\norm{R_Ax-R_Ay}^2-\norm{R_BR_Ax-R_BR_Ay}^2\big)
        \\
        &=\lambda^2\scal{(\Id+R_A)x-(\Id+R_A)y}{(\Id-R_A)x-(\Id-R_A)y}
        \nonumber
        \\
        &\quad+\lambda^2\scal{(\Id+R_B)R_Ax-(\Id+R_B)R_Ay}{
            (\Id-R_B)R_Ax-(\Id-R_B)R_Ay}
        \\
        &=4\lambda^2\scal{J_Ax-J_Ay}{J_{A^{-1}}x-J_{A^{-1}}y}
        +4\lambda^2\scal{J_BR_Ax-J_BR_Ay}{J_{B^{-1}}R_Ax-J_{B^{-1}}R_Ay}.
    \end{align}
\end{subequations}
This proves \cref{e:general:ineq}.

\cref{lem:pre:tele:i}\&\cref{lem:pre:tele:ii}:
Observe that the monotonicity of $A$ (respectively $B$) and the Minty parametrization (see \cref{thm:minty})
of $\gra A$ (respectively $\gra B$) imply that 
$\scal{J_Ax-J_Ay}{J_{A^{-1}}x-J_{A^{-1}}y}\ge 0$
(respectively $
\scal{J_BR_Ax-J_BR_Ay}{J_{B^{-1}}R_Ax-J_{B^{-1}}R_Ay}\ge 0$).
Now combine with  \cref{e:general:ineq}.
\end{proof}	

\begin{corollary}	
\label{cor:limit}
Let $\lambda\in \left]0,1\right]$ and set
\begin{equation}
    T_\lambda =(1-\lambda) \Id +\lambda R_BR_A.
\end{equation}	
Let $(x,y)\in X\times X$.
Then the following hold:
\begin{enumerate}
    \item
    \label{cor:limit:i}
    $\scal{J_A T_\lambda ^nx-J_AT_\lambda ^ny}{J_{A^{-1}}T_\lambda ^nx-J_{A^{-1}}T_\lambda ^ny}\to 0$.
    \item
    \label{cor:limit:ii}
    $\scal{J_BR_AT_\lambda ^nx-J_BR_AT_\lambda ^ny}{J_{B^{-1}}R_AT_\lambda ^nx-J_{B^{-1}}R_AT_\lambda ^ny}\to 0$.
\end{enumerate}
\end{corollary}
\begin{proof}
\cref{cor:limit:i}:
It follows from \cref{lem:pre:tele}\cref{lem:pre:tele:i}
that $(\forall \nnn)$
\begin{equation}
    \norm{T_\lambda ^nx-T_\lambda ^ny}^2 
    - \norm{T_\lambda^{n+1}x-T_\lambda^{n+1}y}^2
    \ge 4\lambda\scal{J_AT_\lambda ^nx-J_AT_\lambda ^ny}{J_{A^{-1}}
        T_\lambda ^nx-J_{A^{-1}}T_\lambda ^ny}.
\end{equation}
Telescoping yields
\begin{equation} 
    \sum_{n=0}^\infty
    \scal{J_AT_\lambda^nx-J_AT_\lambda^ny}{J_{A^{-1}}T_\lambda^nx-J_{A^{-1}}T_\lambda^ny}<+\infty,
\end{equation}
and the conclusion follows using the monotonicity of 
$A$ in view of Minty's parametrization \cref{thm:minty}.
\cref{cor:limit:ii}:
Proceed similar to the proof of \cref{cor:limit:i}
using  \cref{lem:pre:tele}\cref{lem:pre:tele:ii}
and the monotonicity of $B$ in view of \cref{thm:minty}.
\end{proof}	

\section{Dynamic consequences}
\label{sec:4}
Upholding the notation of \cref{section:3}
we recall that \cref{lem:zeros} implies that 
\begin{empheq}[box=\mybluebox]{equation}
\label{e:assump:z}
\text{$Z\neq \fady\siff \fix(\gap+T)\neq \fady\siff\gap\in \ran(\Id-T)$.}
\end{empheq}

\begin{lemma}
\label{lem:gr:A:unif}
Suppose that $A\colon X\rras X$ is uniformly monotone
with modulus $\phi$.
Let $(x,y)\in X\times X$. Then
\begin{equation}
    \scal{J_Ax-J_Ay}{J_{A^{-1}}x-J_{A^{-1}}y}\ge \phi({\norm{J_A x-J_Ay}}).
\end{equation}	
\end{lemma}	
\begin{proof}
This is a direct consequence of
Minty's theorem
\cref{thm:minty}.	
\end{proof}	

Let $C\colon X\rras X$ be uniformly monotone
with modulus $\phi$
and suppose that $\zer C\neq \fady$.
Then $C$ is strictly monotone and it follows from,
e.g., \cite[Proposition~23.35]{BC2017} that 
\begin{equation}
\label{eq:zer:uniq}
\text{$\zer C$ is a singleton}.
\end{equation}	
Moreover, it follows from \cite[Example~25.15(iii)]{BC2017}
that 
\begin{equation}
\label{eq:unifmono:3*}
\text{$\phi$ is supercoercive $\RA$ $C$ is $3^*$ monotone}.
\end{equation}	

We are now ready for the main results of this section.
\begin{theorem}[{\bf convergence of Douglas--Rachford algorithm}]
\label{thm:DR:convergence}
Let $\lambda\in \left]0,1\right[$ and set
\begin{equation}
    T_\lambda =(1-\lambda) \Id +\lambda R_BR_A.
\end{equation}	
Suppose that $Z\neq \fady $ and that  $(\exists C\in \{A,B\})$
such that $C$ is uniformly monotone with a supercoercive modulus.
Then $(\exists \overline{x}\in X )$
such that $Z=\{\overline{x}\}$.
Let $x\in X$. 
Then the following hold:
\begin{enumerate}	
    \item
    \label{thm:DR:convergence:i}	
    $J_AT_\lambda^n x\to \overline{x}$.
    \item
    \label{thm:DR:convergence:ii}
    $J_BR_AT_\lambda^n x\to \overline{x}-\gap$.
\end{enumerate}	
\end{theorem}	

\begin{proof}
To lighten the notation, throughout the proof we set $T=T_\lambda$.  
By assumption it is straightforward to verify that
 $(\exists \overline{C}\in \{-\gap+A,B(\cdot+\gap)\})$
such that $\overline{C}$ is uniformly monotone
and $3^*$ monotone (see \cref{eq:unifmono:3*}).
Hence,
$-\gap+A+B(\cdot+\gap)$ is uniformly monotone and
therefore $Z$ is  a singleton by \cref{eq:zer:uniq}
applied with $C$ replaced by $-\gap+A+B(\cdot+\gap)$.
Now combine with \cref{e:assump:z} to learn that
$\fix(\gap+T)\neq \fady$.
\cref{thm:DR:convergence:i}\&\cref{thm:DR:convergence:ii}:
Indeed, let $y\in \fix(\gap+T)$ and observe that
\cref{fact:fix:p}\cref{fact:fix:p:ii} implies that 
$(\forall \nnn)$ $J_A T^n y=\overline{x}$.
First suppose that $C=A$.
Combining \cref{cor:limit}\cref{cor:limit:i} and \cref{lem:gr:A:unif} applied with 
$(x,y)$ replaced by $(T^n x, T^ny)$ 
we learn that $\phi(\norm{J_AT^n x-\overline{x}})
=\phi(\norm{J_AT^n x-J_AT^ny})\to 0$.
Hence $J_AT^n x\to \overline{x}$. Now combine with 
\cref{fact:fix:p}\cref{fact:fix:p:0} 
to prove \cref{thm:DR:convergence:ii}.
For the case $C=B$, proceed similar to above but use 
\cref{fact:fix:p}\cref{fact:fix:p:iii} and 
\cref{cor:limit}\cref{cor:limit:ii} instead. 
\end{proof}

\begin{theorem}[{\bf convergence of Peaceman--Rachford algorithm}]
\label{thm:PR:convergence}
Set
\begin{equation}
    \widetilde{T}=R_BR_A.
\end{equation}	
Suppose that $Z\neq \fady $ and that  $A$ is uniformly monotone
with supercoercive modulus.
Then $(\exists \overline{x}\in X )$
such that $Z=\{\overline{x}\}$.
Let $x\in X$. 
Then 
$J_A\widetilde{T}^n x\to \overline{x}$.
\end{theorem}	

\begin{proof}
Observe that, likewise $A$, $-\gap+A$
is uniformly monotone and $3^*$ monotone (see \cref{eq:unifmono:3*}).
Hence,
$-\gap+A+B(\cdot+\gap)$ is uniformly monotone and
therefore  $Z$ is  a singleton by applying \cref{eq:zer:uniq}
with $C$ replaced by $-\gap+A+B(\cdot+\gap)$.
Now combine with \cref{e:assump:z} to learn that
$\fix(\gap+T)\neq \fady$.
Let $y\in \fix(\gap+T)$ and observe that \cref{fact:fix:p}\cref{fact:fix:p:ii}  implies
that $(\forall \nnn)$
$J_AT^ny=\overline{x}$.
Therefore \cref{cor:limit}\cref{cor:limit:i} 
implies  that
\begin{equation}
    \label{eq:PR:lim}
    \scal{J_A \widetilde{T}^nx-\overline{x}}{
        J_{A^{-1}}\widetilde{T}^nx-J_{A^{-1}}\widetilde{T}^ny}\to 0.
\end{equation}	
Combining \cref{eq:PR:lim} with \cref{lem:gr:A:unif}	
applied with $x$ replaced by $\widetilde{T}^n x $ yields
$\phi(	\norm{J_A \widetilde{T}^nx-\overline{x}}) \to 0$,			
and the conclusion follows.
\end{proof}

Before proceeding to the illustrative example in this section
we recall the following fact.
\begin{fact}
\label{fact:ran:C}
Suppose that $(\exists C\in \{A,B\})$
such that $C$ is $3^*$ monotone and 
surjective. Then 
\begin{equation}
    \cran(\Id-T)=\overline{(\dom A-\dom B)
        \cap  (\ran A+\ran B)}=\overline{\dom A-\dom B}.
\end{equation}
\end{fact}
\begin{proof}
See \cite[Theorem~2.5]{WMM2024}
and also \cite[Theorem~5.2]{MOR}.
\end{proof}
\begin{example}
\label{ex:two:balls}
Suppose that $X=\RR^2$,
let $\theta\in \left[0,\tfrac{\pi}{2}\right[$,
let $(\beta,\gamma )\in \RR^2$,
let $r>0$,
set $b=(0,\beta)$,
set $c=(\gamma,0)$,
set 
\begin{equation}
    \rot=\rot_\theta=	\begin{pmatrix}
        \cos(\theta) &-\sin (\theta) 
        \\
        \sin (\theta) &	\cos(\theta)
    \end{pmatrix},
\end{equation}
set $A=\rot+N_{\ball{0}{1}}$
and set $B=b+N_{\ball{c}{r}}$.
Let $x\in X$, let $\lambda\in \left]0,1\right]$
and set 
\begin{equation}
    T_\lambda=(1-\lambda)\Id +\lambda R_BR_A
    =\Id-2\lambda J_A+2\lambda J_B(2J_A-\Id).
\end{equation}	
Then 
\begin{equation}
    \label{ex:JA}
    J_A	x=
    \begin{cases}
        \tfrac{1}{2(1+\cos(\theta)) }(\Id+\rot_{-\theta})x,
        &\norm{x}^2\le 2(1+\cos\theta);
        \\
        \tfrac{1}{\norm{x}^2}\Big(\Big(\sqrt{\norm{x}^2
            - \sin^2(\theta) }-\cos(\theta)\Big)\Id+\rot_{-\theta}\Big)x,
        &\text{otherwise},
    \end{cases}
\end{equation}	
\begin{equation}
    \label{ex:JB}
    J_Bx=P_{\ball{c}{r}}(x-b)	
    =\begin{cases}
        x-b,
        &\norm{x-(b+c)}\le r;
        \\
        c+r\frac{x-(b+c)}{\norm{x-(b+c)}},
        &\text{otherwise}.
    \end{cases}
\end{equation}

Moreover, we have:
\begin{enumerate}
    \item 
    \label{ex:two:balls:i}
    $A$ is maximally monotone and
    strongly monotone, hence $A$ is uniformly monotone.
    \item
    \label{ex:two:balls:ii}
    $B$ is maximally monotone.
    \item
    \label{ex:two:balls:iii}
    $\cran(\Id-T)=\overline{\dom A-\dom B}={\dom A-\dom B}=\ball{-c}{r+1}$.
    \item
    \label{ex:two:balls:iv}
    $\gap=(\sign (\gamma)\min\{0,r+1-\abs{\gamma}\},0)$.
    \item
    \label{ex:two:balls:v}
    Exactly one of the following holds:
    \begin{enumerate}
        \item 
        \label{ex:two:balls:v:a}
        $\abs{\gamma}<(r+1)$ and $A+B$
        is maximally monotone,
        in which case $\gap= (0,0)\in \ran (\Id-T)$ 
        and $Z$ is a singleton. 
        \item
        \label{ex:two:balls:v:b}
        $\abs{\gamma}\ge (r+1)$ and 
        $\gap= (\sign (\gamma)(r+1)-{\gamma},0)$, in which case\\[1mm]
        $\left[Z
        \neq \fady\siff Z=\{\sign(\gamma)(1,0)\}
        \siff
        \beta=-\sign(\gamma)\sin \theta \right]$.
        
    \end{enumerate}	
    \item
    \label{ex:two:balls:vi}
    Suppose that $\beta=-\sign(\gamma)\sin \theta$.
    Then $Z=\{\sign(\gamma)(1,0)\}$
    and the following hold:
    \begin{enumerate}
        \item 
        \label{ex:two:balls:vi:c}
        $J_AT_\lambda^n x\to \sign(\gamma)(1,0)$.
        \item
        \label{ex:two:balls:vi:d}
        Suppose that $\lambda\in \left]0,1\right[$. Then 
        $J_BR_AT_\lambda^n x\to  (\gamma-\sign(\gamma)r,0)$.
    \end{enumerate}	
    
    
\end{enumerate}

\end{example}	
\begin{proof}
The formula in  \cref{ex:JA}	
follows from applying \cref{ex:Id:S} with $(\alpha,\beta,\gamma)$
replaced by $(\cos\theta,\sin \theta,1)$
and $S$ replaced by 
$\begin{psmallmatrix}
    0&-1\\
    1&0
\end{psmallmatrix}$.
The first identity in 
\cref{ex:JB} follows from 
\cite[Proposition~23.17(ii)]{BC2017}
whereas the second identity follows from , e.g., 
\cite[Proposition~29.10]{BC2017}.

\cref{ex:two:balls:i}:	
It is straightforward to verify that 
$A$ is $\cos \theta$-strongly monotone.
The maximal monotonicity of 
$A$ follows from, e.g., \cite[Corollary~25.5(i)]{BC2017}.

\cref{ex:two:balls:ii}:	
This is clear by observing that
$B=\partial(\iota_{\ball{c}{r}}+\scal{b}{\cdot})$.

\cref{ex:two:balls:iii}:	
On the one hand, we have 
$\dom A=\ball{0}{1}$
and $\dom B=\ball{c}{r}$.
Therefore $\ran A=\ran B=\RR^2$ by, e.g., 
\cite[Corollary~21.15]{BC2017}.
On the other hand, 
combining \cref{ex:two:balls:i}
and \cite[Example~25.15(iv)]{BC2017}
we learn that
$A$ is $3^*$ monotone. 
Moreover, $B$ is $3^*$ monotone
by, e.g., \cite[Example~25.14]{BC2017}. 
Now, apply \cref{fact:ran:C} with $C$
replaced by either $A$ or $B$
to learn that 
$\cran(\Id-T)=\ball{0}{1}-\ball{c}{r}=\ball{-c}{r+1}$.

\cref{ex:two:balls:iv}:	
Observe that $-c=(-\gamma,0)$.
If $\abs{\gamma}<r+1$ then $(0,0)\in \ball{-c}{r+1}$
and hence $\gap=(0,0)$.
Else, if $ \abs{\gamma}\ge r+1$ then $\gap=(\sign(\gamma)(r+1)-\gamma,0)$
and the conclusion follows.

\cref{ex:two:balls:v:a}:	
Indeed, in view of \cref{ex:two:balls:iii} and \cref{ex:two:balls:iv} 
we have 
$(0,0) \in \intr \ball{-c}{r+1}=\intr (\dom A-\dom B)$.
Therefore, $A+B$ is maximally monotone by, e.g.,
\cite[Corollary~25.5(iii)]{BC2017}
and strongly monotone by \cref{ex:two:balls:i}.
Hence, $\zer(A+B) $ is a singleton by, 
e.g., \cite[Corollary~23.37]{BC2017}.

\cref{ex:two:balls:v:b}:	
The formula for $\gap$
follows from \cref{ex:two:balls:iv}.
Observe that $Z\subseteq\dom(-\gap+A)\cap\dom B(\cdot-\gap)
=\ball{0}{1}\cap\ball{c+\gap}{r}
=\{\sign(\gamma)(1,0)\}$.
This proves the first equivalence.
We now turn to the second equivalence.
Indeed, set $z=\sign(\gamma)(1,0)$.
Then 
$0\in -\gap+Rz+N_{\ball{0}{1}}z+b+N_{\ball{c}{r}}z
\siff 0\in -\gap+Rz+b+\RR\times \{0\}=Rz+b+\RR\times \{0\}
\siff \sign(\gamma)(\cos \theta ,\sin \theta)+(0,\beta )\in\RR\times \{0\}\siff 
\beta=-\sign(\gamma)\sin \theta$.

\cref{ex:two:balls:vi:c}:	
Combine \cref{ex:two:balls:v:b},
\cref{thm:DR:convergence}\cref{thm:DR:convergence:i}
and \cref{thm:PR:convergence}.
\cref{ex:two:balls:vi:d}:	
Combine \cref{ex:two:balls:v:b} and 
\cref{thm:DR:convergence}\cref{thm:DR:convergence:ii}.
\end{proof}	

\begin{figure}
\begin{center}
\includegraphics[scale=0.14]{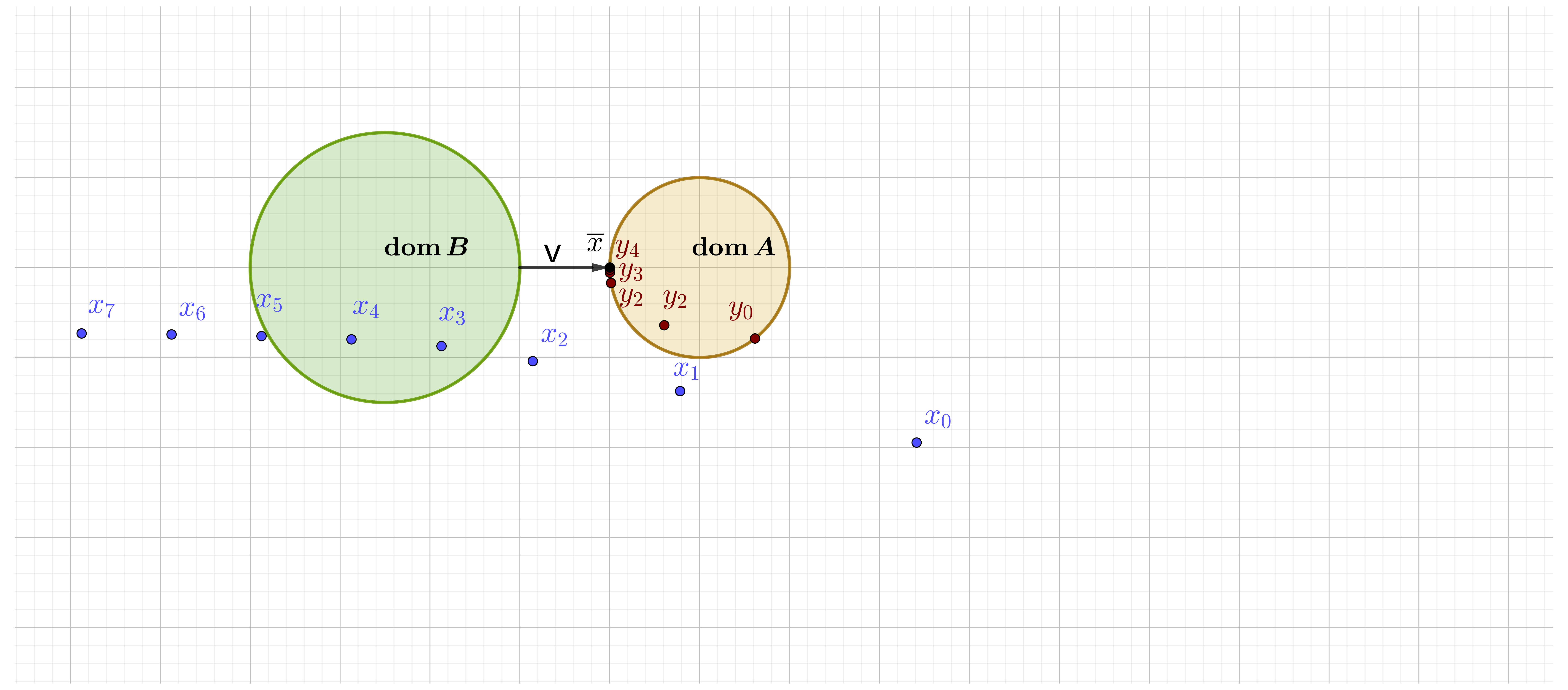}
\end{center}	
\caption{A \texttt{GeoGebra} \cite {geogebra} 
snapshot illustrating
\cref{ex:two:balls}\cref{ex:two:balls:vi} with $\theta={\pi}/{4}$,
$\gamma=-3.5$,
$r=3/2$
and 
$\lambda=1/2$.
The first few iterates of the governing sequence 
$(x_n)_\nnn=(T^n x)_\nnn$ (the blue  dots)
and the shadow sequence $(y_n)_\nnn=(J_AT^n x)_\nnn$ (the red dots)
are also depicted.}
\end{figure}

\section{Application to optimization problems}
\label{sec:5}
In this section we assume that
\begin{empheq}[box=\mybluebox]{equation}
\text{$f$ and $g$ are proper lower semicontinuous convex functions on $X$,}
\end{empheq}
and we set 
\begin{empheq}[box=\mybluebox]{equation}
\label{eq:AB:fg}
(A,B)=(\partial f,\partial g).
\end{empheq}
We use the abbreviations
\begin{empheq}[box=\mybluebox]{equation}
(\Rr_f,\Rr_g)=(2\prox_f-\Id,2\prox_g-\Id).
\end{empheq}
Hence,
\begin{equation}
T=T_{(\partial f,\partial g)}=\tfrac{1}{2}(\Id+\Rr_g\Rr_f)=\Id-\pr_f+\pr_g\Rr_f,
\end{equation}	
and 
\begin{equation}
Z=\menge{x\in X}{0\in -\gap+\partial f(x)+\partial g(x-\gap)}.
\end{equation}	

\begin{lemma}
\label{lem:fg}
Suppose that $(\exists h\in \{f,g\})$
such that $h$ is uniformly convex on $X$.
Then the following hold:
\begin{enumerate}
    \item 
    \label{lem:fg:i}
    $	\cran(\Id-T)=\overline{\dom f -\dom g}$.
    \item	
    \label{lem:fg:ii}
    $	\gap=P_{\overline{\dom f -\dom g}}0$.
    \item
    \label{lem:fg:iii}
    $\partial h$ is uniformly monotone.
\end{enumerate}	
Suppose that $Z\neq \fady $. 
Then there exists 
$\overline{x}\in X$ such that
the following hold:
\begin{enumerate}
    \setcounter{enumi}{3}
    \item
    \label{lem:fg:ii:iii}
    $-\scal{\cdot}{\gap}+f+g(\cdot-\gap)$ is uniformly convex,
    lower semicontinous and proper.
    \item
    \label{lem:fg:iv}
    $Z=\argmin(-\scal{\cdot}{\gap}+f+g(\cdot-\gap))=\{\overline{x}\}	$.
\end{enumerate}	
\end{lemma}	
\begin{proof}
\cref{lem:fg:i}:
It follows from \cite[Theorem~3.5.5(i)]{Za02}
that $h^*$ is uniformly smooth, hence
 $\dom  h^*=X$.
Now combine this with 
\cite[Theorem~3.3(i)]{WMM2024}.

\cref{lem:fg:ii}:
Combine \cref{lem:fg:i} and  \cref{eq:def:v}.

\cref{lem:fg:iii}:
This is \cite[Example~22.4(iii)]{BC2017}.

\cref{lem:fg:ii:iii}:
It is straightforward to verify that
$(\exists \overline{h}\in \{-\scal{\cdot}{\gap}+f,g(\cdot-\gap)\})$
such that $\overline{h}$ is uniformly convex 
and therefore $-\scal{\cdot}{\gap}+f+g(\cdot-\gap)$
is uniformly convex .
The lower semicontinuity 
of $-\scal{\cdot}{\gap}+f+g(\cdot-\gap)$
is a direct consequence of 
the lower semicontinuity of $f$ and $g$.
Finally observe that 
$\fady\neq Z\subseteq \dom f\cap \dom g(\cdot-\gap)$.
Hence, $-\scal{\cdot}{\gap}+f+g(\cdot-\gap)$ is proper.

\cref{lem:fg:iv}:
Indeed, recall that 
$Z=\menge{x\in X}{0\in -\gap+\partial f(x)+\partial g(x-\gap)}$.
On the one hand, it follows from \cref{lem:fg:iii}
that $-\gap+\partial f(x)+\partial g(x-\gap)$ is uniformly monotone.
On the other hand,
\cref{eq:zer:uniq} applied with $C$
replaced by $-\gap+\partial f(x)+\partial g(x-\gap)$
implies that $Z$ is a singleton.
Now observe that 
$Z\subseteq \argmin(-\scal{\cdot}{\gap}+f+g(\cdot-\gap))$,
and the latter set is a singleton by combining \cref{lem:fg:ii:iii} and, 
e.g., \cite[Proposition~17.26(iii)]{BC2017}.
\end{proof}	

The following fact will be used throughout the examples in this section to calculate $\gap$.

\begin{fact}
\label{fact:pazybr}
Let $x_0\in X$ and $\lambda\in\left]0,1\right ]$. Update $x_0$ via $x_{n+1}=T_\lambda x_n$ where $T_\lambda$ is as in \cref{eqn:T}. The following hold.
\begin{enumerate}
\item 
\label{eqn:v1}
$\dfrac{x_n}{n} \rightarrow -2\lambda \gap$.
\item 
\label{eqn:v2}
Suppose $\lambda\in\left]0,1\right [$.
Then $x_n-x_{n+1}\rightarrow 2\lambda\gap$.
\end{enumerate}
\end{fact}
\begin{proof}
See \cite{Pazy1970} for \cref{eqn:v1}
and  \cite{BBR78} or \cite{BR77} for \cref{eqn:v2}.
\end{proof}


\begin{theorem}[{\bf convergence of Douglas--Rachford algorithm}]
\label{thm:DR:convergence:fg}
Let $\lambda\in \left]0,1\right[$ and set
\begin{equation}
    T_\lambda =(1-\lambda) \Id +\lambda \Rr_g\Rr_f.
\end{equation}	
Suppose that $Z\neq \fady$ and that 
$(\exists h\in \{f,g\})$
such that $h$ is uniformly convex on $X$.
Then $(\exists \overline{x}\in X )$
such that $Z=\argmin(-\scal{\cdot}{\gap}+f+g(\cdot-\gap))=\{\overline{x}\}$.
Let $x\in X$. 
Then the following hold:
\begin{enumerate}	
    \item
    \label{thm:DR:convergence:fg:i}	
    $\pr_fT_\lambda^n x\to \overline{x}$.
    \item
    \label{thm:DR:convergence:fg:ii}
    $\pr_g\Rr_fT_\lambda^n x\to \overline{x}-\gap$.
\end{enumerate}	
\end{theorem}	

\begin{proof}
\cref{thm:DR:convergence:fg:i}\&\cref{thm:DR:convergence:fg:ii}:		
Combine 
\cref{thm:DR:convergence}\cref{thm:DR:convergence:i}\&\cref{thm:DR:convergence:ii},
applied with $(A,B,C)$ replaced by
$(\partial f, \partial g,\partial h)$,
and \cref{lem:fg}\cref{lem:fg:iii}\&\cref{lem:fg:iv}.
\end{proof}	

\begin{theorem}[{\bf convergence of Peaceman--Rachford algorithm}]
\label{thm:PR:convergence:fg}
Set
\begin{equation}
    \widetilde{T}=\Rr_g\Rr_f.
\end{equation}	
Suppose that  $Z\neq \fady$ and that  $f$ is uniformly convex on $X$.
Then $(\exists \overline{x}\in X )$
such that $Z=\argmin(-\scal{\cdot}{\gap}+f+g(\cdot-\gap))=\{\overline{x}\}$.
Let $x\in X$. 
Then 
$\pr_f\widetilde{T}^n x\to \overline{x}$.
\end{theorem}	

\begin{proof}
Combine 
\cref{thm:PR:convergence}
applied with $(A,B,C)$ replaced by
$(\partial f, \partial g,\partial h)$,
and \cref{lem:fg}\cref{lem:fg:iii}\&\cref{lem:fg:iv}.	
\end{proof}	

\begin{proposition}
\label{prop:example:fg}
Suppose that $X$ is finite-dimensional,
let $w\in X$ and let $U$ and $V$ be 
nonempty  polyhedral subsets\footnote{A subset $U$ of $X$
    is a polyhedral if it is a finite intersection of closed halfspaces.} of $X$.
Let $r>0$ and set $f=\tfrac{r}{2}\norm{\cdot-w}^2+\iota_U$
and set $g=\iota_V$. Let $x\in X$.
Then $(\exists \overline{x}\in X)$ such that the following holds:
\begin{enumerate}
    \item
    \label{prop:example:fg:i}
    $f$ is strongly convex. 	
    \item
    \label{prop:example:fg:ii}
    $\cran(\Id-T)={U-V}$.
    \item
    \label{prop:example:fg:iii}
    $\gap \in U-V=\dom f-\dom g$.
    \item
    \label{prop:example:fg:iv}
    $Z
    =\argmin(-\scal{\cdot}{\gap}
    +\tfrac{r}{2}\norm{\cdot-w}^2+\iota_U+\iota_V(\cdot-\gap))
    =\{\overline{x}\}$.
    \item
    \label{prop:example:fg:v}
    $\gap\in \ran (\Id-T)$.
    \item
    \label{prop:example:fg:vi}
    $\pr_fx=P_U\left(\tfrac{r}{r+1}x+\tfrac{1}{r+1}w\right)$.
    \item
    \label{prop:example:fg:vii}
    $\pr_gx=P_Vx$.
    \item
    Let $\lambda\in \ocint{0,1}$ and 
    set $T_\lambda=(1-\lambda)\Id+\lambda \Rr_g\Rr_f$. 
    Then we have: 
    \begin{enumerate}
        \item
        \label{prop:example:fg:viii:a}
        $\pr_f T_\lambda^nx\to \overline{x}$.
        \item
        \label{prop:example:fg:viii:b}
        Suppose that $\lambda\in \opint{0,1}$.
        Then $\pr_g\Rr_f T_\lambda^nx\to \overline{x}-\gap$.
    \end{enumerate}	
\end{enumerate}	
\end{proposition}	
\begin{proof}
\cref{prop:example:fg:i}:
This is a direct consequence of 
the strong convexity of $\tfrac{r}{2}\norm{x-w}^2$.

\cref{prop:example:fg:ii}\&\cref{prop:example:fg:iii}:
Observe that  $\dom f-\dom g=U-V $ is polyhedral 
by \cite[Corollary~19.3.2]{Rock70}	
hence closed.
Now combine with \cref{prop:example:fg:i} and 
\cref{lem:fg}\cref{lem:fg:i}\&\cref{lem:fg:ii}.

\cref{prop:example:fg:iv}:
It follows from \cref{prop:example:fg:iii} that
$\dom N_U\cap \dom N_{V}(\cdot-\gap)=U\cap(\gap+V)\neq \fady$.
Hence  $N_U+ N_{V}(\cdot-\gap)=\partial(\iota_U+\iota_V(\cdot+\gap))$ 
by, e.g.,
\cite[Theorem~16.47(iii)]{BC2017},
and therefore is maximally monotone.
Applying \cref{thm:minty}  with $A$
replaced by $\tfrac{1}{r}(-r w+N_U+ N_{V}(\cdot-\gap) )$
yields that $-\gap+r \Id-r w+N_U+ N_{V}(\cdot-\gap)$
is surjective, hence $Z\neq \fady$.
Now combine with \cref{lem:fg}\cref{lem:fg:iv}.

\cref{prop:example:fg:v}:
Combine \cref{prop:example:fg:iv}
and \cref{e:assump:z}.

\cref{prop:example:fg:vi}\&\cref{prop:example:fg:vii}:
This is clear by, e.g., \cite[Examples~23.3\&23.4]{BC2017}.

\cref{prop:example:fg:viii:a}:
Apply \cref{thm:DR:convergence:fg}\cref{thm:DR:convergence:fg:i}
(for the case $\lambda\in \opint{0,1}$)
and  \cref{thm:PR:convergence:fg} (for the case 
$\lambda =1$)
in view of  \cref{prop:example:fg:i} and 
\cref{prop:example:fg:iv}.

\cref{prop:example:fg:viii:b}:
Apply \cref{thm:DR:convergence:fg}\cref{thm:DR:convergence:fg:ii}
in view of  \cref{prop:example:fg:i} and 
\cref{prop:example:fg:iv}.
\end{proof}

\begin{example}
\label{example:fg}
Suppose that $X=\RR^2$, let $(\alpha_i,\beta_i)\in \RR^2$
be such that $\alpha_i\le \beta_i$, $i\in \{1,2\}$,
let $U=\RR\times \{0\}$, let $w=(w_1,0)\in U$ and let 
$V=\clint{\alpha_1,\beta_1}\times \clint{\alpha_2,\beta_2}$.
Set $f=\tfrac{1}{2}\norm{\cdot-w}^2+\iota_U$
and set $g=\iota_V$. Let $x=(x_1,x_2)\in \RR^2$.
Then  the following hold:
\begin{enumerate}
    \item
    \label{example:fg:i}
    $f$ is strongly convex. 	
    \item
    \label{example:fg:ii}
    $\cran(\Id-T)=U-V=\RR \times \clint{-\beta_2,-\alpha_2}$.
    \item
    \label{example:fg:iii}
    $\gap \in U-V=\dom f-\dom g$.
    \item
    \label{example:fg:iv}
    $Z
    =\argmin(\tfrac{1}{2}\norm{\cdot-w}^2+\iota_U+\iota_V(\cdot-\gap))
    =\{ (P_{\clint{\alpha_1,\beta_1}}w_1,0)\}$.
    \item
    \label{example:fg:v}
    $\gap=(0,\min\{\max\{0,-\beta_2\},-\alpha_2\})
    \in \ran (\Id-T)$.
    \item
    \label{example:fg:vi}
    $\pr_fx=\tfrac{r}{r+1}P_Ux+\tfrac{1}{r+1}w$.
    \item
    \label{example:fg:vii}
    $\pr_gx=P_Vx$, where $(P_Vx)_i=\min\{\max\{x_i,\alpha_i\},\beta_i\}$.
    \item
    Let $\lambda\in \ocint{0,1}$ and 
    set $T_\lambda=(1-\lambda)\Id+\lambda \Rr_g\Rr_f$. 
    Then we have: 
    \begin{enumerate}
        \item
        \label{example:fg:viii:a}
        $\pr_f T_\lambda^nx\to (P_{\clint{\alpha_1,\beta_1}}w_1,0)$.
        \item
        \label{example:fg:viii:b}
        Suppose that $\lambda\in \opint{0,1}$.
        Then $\pr_g\Rr_f T_\lambda^nx\to (P_{\clint{\alpha_1,\beta_1}}w_1,
        \max\{\min\{0,\beta_2\},\alpha_2\})$.
    \end{enumerate}	
\end{enumerate}	
\end{example}
\begin{proof}
\cref{example:fg:i}--\cref{example:fg:iii}:
This is 
\cref{prop:example:fg}\cref{prop:example:fg:i}--\cref{prop:example:fg:iii}
applied with $X$ replaced by $\RR^2$
and $r=1$ 
by observing that $U$ and $V$ are nonempty 
polyhedral subsets of $\RR^2$. 

\cref{example:fg:iv}:
Indeed, 	
\begin{subequations}
    \begin{align}
        Z&=\menge{x\in\RR^2}{0\in -\gap+x-w+N_Ux+N_V(x-\gap)}	
        \\
        &=\menge{x\in\RR^2}{0\in -\gap+x-w+U^{\perp}+N_V(x-\gap)}
        \\
        &=\menge{x\in\RR^2}{0\in x-w+U^\perp+N_V(x-\gap)}
        \label{se:c}
        \\
        &\subseteq
        \argmin(\tfrac{1}{2}\norm{x-w}^2+\iota_U+\iota_V(\cdot-\gap)),
    \end{align}
\end{subequations}	
where \cref{se:c} follows from combining
\cref{example:fg:iii} and \cite[Remark~2.8(ii)]{Lukepaper}.
Now on the one hand, \cref{prop:example:fg}\cref{prop:example:fg:iv}
implies that $Z$ is a singleton. 
On the other hand,  
the strong convexity of $\tfrac{1}{2}\norm{x-w}^2$ guarantees that
$	\argmin(\tfrac{1}{2}\norm{x-w}^2+\iota_U+\iota_V(\cdot-\gap))$
is a singleton.		Altogether, this verifies the 
first identity in 		\cref{example:fg:iv}.
We now turn to the second identity.
Set $\overline{x}=(P_{\clint{\alpha_1,\beta_1}}w_1,0)\in U$
and 
observe that by \cref{se:c} we have 
$\{\overline{x}\}= Z\siff w-\overline{x}\in U^\perp +N_{V}(\overline{x}-\gap)$.
We examine three cases.
\textsc{Case~1:} 
$w_1\in \clint{\alpha_1,\beta_1}$. Then 
$w-\overline{x}=(0,0)\in U^\perp +N_{V}(\overline{x}-\gap)$. 
\textsc{Case~2:}   $w_1<\alpha_1$.
In this case $N_{V}(\overline{x}-\gap)\in 
\{\RR_{-}\times \RR_{+},\RR_{-}\times \RR_{-}\}$, hence 
$w-\overline{x}\in \RR_{--}\times\{0\}
\subseteq \RR_{-}\times\RR
= U^\perp +N_{V}(\overline{x}-\gap)$. 
\textsc{Case~3:}   $w_1>\alpha_1$.
Proceed similar to \textsc{Case~2}.   

\cref{example:fg:v}:
This is a direct consequence of \cref{example:fg:ii}. 

\cref{example:fg:vi}:
Combine \cref{prop:example:fg}\cref{prop:example:fg:vi} and \cite[5.13(i)]{Deutsch}.

\cref{example:fg:vii}:
This follows from, e.g., \cite[Lemma~6.26]{Beck2017}.

\cref{example:fg:viii:a}\&\cref{example:fg:viii:b}:
Combine \cref{example:fg:iv}, \cref{example:fg:v}   and \cref{prop:example:fg}\cref{prop:example:fg:viii:a}\cref{prop:example:fg:viii:b}.
\end{proof}	

\begin{figure}
\begin{center}
\includegraphics[scale=0.95]{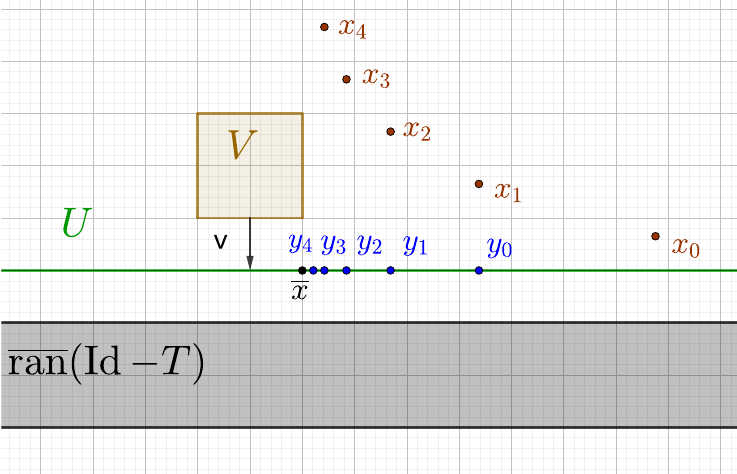}
\end{center}	
\caption{A \texttt{GeoGebra} \cite{geogebra} snapshot illustrating
\cref{example:fg} with 
$\lambda=1/2$,
$(\alpha_1,\beta_1) = (-1,1)$,
$(\alpha_2,\beta_2) = (1,3)$
and 
$w=(1,0)$.
The first few iterates of the governing sequence 
$(x_n)_\nnn=(T^n x)_\nnn$ (the red  dots)
and the shadow sequence $(y_n)_\nnn=(J_AT^n x)_\nnn$ (the blue dots)
are also depicted.}
\end{figure}	

\begin{example}
\label{example:orth}
Suppose that $X=\RR^d$ and let $x=(x_1,\dots,x_d)\in X$. Let $U=\RR_+^d$ and $V=\menge{x\in X}{Lx=b}$ where $L\in\RR^{m\times d}$ has full rank, $b\in\RR^m$, $m\leq d$.
Set $f=\tfrac{r}{2}\norm{\cdot-z}^2+\iota_U$ where $r>0$, $z\in X$, $\gamma:=\frac{1}{1+r}\in]0,1[$
and set $g=\iota_V$.
Then  the following hold:
\begin{enumerate}
    \item
    \label{example:orth:i}
    $f$ is strongly convex. 	
    \item
    \label{example:orth:ii}
    $\cran(\Id-T)=U-V.$
    \item
    \label{example:orth:iii}
    $\gap \in U-V=\dom f-\dom g$.
    \item
    \label{example:orth:iv}
    $Z
    =\argmin(\tfrac{r}{2}\norm{\cdot-w}^2+\iota_U+\iota_V(\cdot-\gap))=\{\overline{x}\}$.
    \item
    \label{example:orth:v}
    $\gap \in \ran (\Id-T)$.
    \item
    \label{example:orth:vi}
    $\pr_fx=P_U(\gamma x+(1-\gamma)z)$, where $(P_Ux)_i=\min\{x_i,0\}$.
    \item
    \label{example:orth:vii}
    $\pr_gx=P_Vx=x-L\tran(LL\tran)^{-1}(Lx-b)$.
    \item
    Let $\lambda\in \ocint{0,1}$ and 
    set $T_\lambda=(1-\lambda)\Id+\lambda \Rr_g\Rr_f$. 
    Then we have: 
    \begin{enumerate}
        \item
        \label{example:orth:viii:a}
        $\pr_f T_\lambda^nx\to \overline{x}$.
        \item
        \label{example:orth:viii:b}
        Suppose that $\lambda\in \opint{0,1}$.
        Then $\pr_g\Rr_f T_\lambda^nx\to \overline{x}-\gap$.
    \end{enumerate}	
\end{enumerate}	
\end{example}
\begin{proof}
\cref{example:box:i}--\cref{example:box:iii}:
This is 
\cref{prop:example:fg}\cref{prop:example:fg:i}--\cref{prop:example:fg:iii}
applied with $X$ replaced by $\RR^d$
by observing that $U$ and $V$ are nonempty 
polyhedral subsets of $\RR^d$.

\cref{example:box:iv}--\cref{example:box:v}:
This is \cref{prop:example:fg}\cref{prop:example:fg:iv}--\cref{prop:example:fg:v}.

\cref{example:box:vi}:
This follows from \cite[Lemma~6.26]{Beck2017} and \cite[Prop.~24.8(i)]{BC2017}.

\cref{example:box:vii}
This follows from \cite[Lemma~6.26]{Beck2017}.

\cref{example:box:viii:a}--\cref{example:box:viii:b}
This is \cref{prop:example:fg}\cref{prop:example:fg:viii:a}--\cref{prop:example:fg:viii:b}.
\end{proof}

\begin{example}
\label{example:box}
Suppose that $X=\RR^d$ and let $x=(x_1,\dots,x_d)\in X$. Let $U=\menge{x\in X}{\alpha_i\leq x_i \leq\beta_i}$
such that $0\le\alpha_i\le \beta_i$ for $i=1,\dots,d$ and let $V=\menge{x\in X}{Lx=b}$ where $L\in\RR^{m\times d}$ has full rank, $b\in\RR^m$, $m\leq d$.
Set $f=\tfrac{r}{2}\norm{\cdot-z}^2+\iota_U$ where $r>0$, $\gamma:=\frac{1}{1+r}\in]0,1[$, $z\in X$
and set $g=\iota_V$.
Then  the following hold:
\begin{enumerate}
    \item
    \label{example:box:i}
    $f$ is strongly convex. 	
    \item
    \label{example:box:ii}
    $\cran(\Id-T)=U-V.$
    \item
    \label{example:box:iii}
    $\gap \in U-V=\dom f-\dom g$.
    \item
    \label{example:box:iv}
    $Z
    =\argmin(\tfrac{r}{2}\norm{\cdot-w}^2+\iota_U+\iota_V(\cdot-\gap))=\{\overline{x}\}$.
    \item
    \label{example:box:v}
    $\gap \in \ran (\Id-T)$.
    \item
    \label{example:box:vi}
    $\pr_fx=P_U(\gamma x+(1-\gamma)z)$, where $(P_Ux)_i=\min\{\max\{x_i,\alpha_i\},\beta_i\}$.
    \item
    \label{example:box:vii}
    $\pr_gx=P_Vx=x-L\tran(LL\tran)^{-1}(Lx-b)$.
    \item
    Let $\lambda\in \ocint{0,1}$ and 
    set $T_\lambda=(1-\lambda)\Id+\lambda \Rr_g\Rr_f$. 
    Then we have: 
    \begin{enumerate}
        \item
        \label{example:box:viii:a}
        $\pr_f T_\lambda^nx\to \overline{x}$.
        \item
        \label{example:box:viii:b}
        Suppose that $\lambda\in \opint{0,1}$.
        Then $\pr_g\Rr_f T_\lambda^nx\to \overline{x}-\gap$.
    \end{enumerate}	
\end{enumerate}	
\end{example}
\begin{proof}
\cref{example:box:i}--\cref{example:box:iii}:
This is 
\cref{prop:example:fg}\cref{prop:example:fg:i}--\cref{prop:example:fg:iii}
applied with $X$ replaced by $\RR^d$
by observing that $U$ and $V$ are nonempty 
polyhedral subsets of $\RR^d$.

\cref{example:box:iv}--\cref{example:box:v}:
This is \cref{prop:example:fg}\cref{prop:example:fg:iv}--\cref{prop:example:fg:v}.

\cref{example:box:vi}:
This follows from \cite[Lemma~6.26]{Beck2017} and \cite[Prop.~24.8(i)]{BC2017}.

\cref{example:box:vii}:
This follows from \cite[Lemma~6.26]{Beck2017}.

\cref{example:box:viii:a}--\cref{example:box:viii:b}:
This is \cref{prop:example:fg}\cref{prop:example:fg:viii:a}--\cref{prop:example:fg:viii:b}.
\end{proof}

\section{Numerical experiments}
\label{sec:num_exp}

This section contains numerical experiments with the inconsistent problems in \cref{example:orth}--\cref{example:box}. We introduce Dykstra's algorithm to compare performance with the Douglas--Rachford algorithm and Peaceman--Rachford algorithm. Following this we give experiments using different values of the algorithmic parameters $\gamma$ and $\lambda$ to determine which choices provide optimal performance, i.e. (i) the smallest number of iterations, (ii) an accurate approximate solution $\overline{x}$, and (iii)~an accurate approximation of the gap vector $\gap$.

\label{sec:6}
Let ${\cal A}, {\cal B}$ be closed convex subsets of $X$ and $z\in X$. Dykstra's projection algorithm (see \cite{BoyleDykstra}) operates as follows:  
Set $a_0 := z,\, p_0 := 0$ and $q_0 := 0$.  Given $a_n,p_n,q_n$, where $n\geq 0$, update
\begin{align*}
&b_{n} := P_{\cal{B}}(a_{n}+q_{n})\,,
\ \ \ \ \ \,\,
q_{n+1}:= a_{n}+q_n-b_{n}\,, \\
&a_{n+1} := P_{\cal{A}}(b_{n}+p_n)\,,
\ \ \ 
p_{n+1} := b_{n}+p_n-a_{n+1}\,.
\end{align*}
It is known that both $(a_n)_{n\in\mathbb{N}}$ and $(b_n)_{n\in\mathbb{N}}$ converge \emph{strongly} (i.e., converge in norm), to $P_{\cal{A}\cap \cal{B}}(z)$.

From \cite[Corollary~3.4]{BB94} we have that the sequences $(a_n)_\nnn, (b_n)_\nnn$ satisfy
\begin{equation}\label{eqn:v_Dyk}
b_n-a_n,\ b_n-a_{n+1} \rightarrow \gap.
\end{equation}

Recall that the set $V$ in \cref{example:orth}--\cref{example:box} is given as $V=\menge{x\in X}{Lx=b}$. In these numerical experiments the $m\times d$ matrix $L$ and $m\times 1$ vector $b$ were generated using uniformly distributed random numbers between $-50$ and $50$. The numerical experiments for \cref{example:orth}--\cref{example:box} with the same dimensions $m,d$ use the same matrix $L$ and vector $b$. To ensure that $U\cap V=\emptyset$ in \cref{example:orth} we used Farkas' Lemma to choose the sign of the columns of $L$ in such a way that we can guarantee that $x\in V$ would not intersect with $U=\mathbb{R}_+^d$. In \cref{example:box} we choose the box $U$ as a subset of $\mathbb{R}_+^d$ to ensure the intersection remains empty. We can write \cref{example:orth}--\cref{example:box} as
\[
\ds\minimize{}{} \ \ \ds \phi(x) := \frac{r}{2}\|x-z\|^2\, \,\mbox{ subject to} \ \ x\in U\cap V\,.
\]

\subsection{Computational algorithms}
Below are the algorithms as implemented in the experiments. Recall that for Peaceman--Rachford $\lambda$ defined in Step 1 below is fixed as 1.
\noindent
\begin{algorithm}{\bf(Douglas--Rachford Relaxation and Peaceman--Rachford)}\label{alg:DR} \rm
\begin{description}
\item[Step 1] ({\em Initialization}) Choose a parameter $r>0$, $\gamma:=1/(r+1)$, $\lambda\in\left]0,1\right]$ and the initial iterate $x_0$ arbitrarily. 
Choose a small parameter $\varepsilon>0$, and set $n=0$. 
\item[Step 2] ({\em Projection onto $U$})  Set\ \ $x^- 
:= x_{n}$. 
Compute\ \ $\widetilde{x} = P_{U}(\gamma x^-+(1-\gamma)z)$\ \ using \cref{example:orth}\cref{example:orth:vi} or \cref{example:box}\cref{example:box:vi}. 
\item[Step 3] ({\em Projection onto $V$})
Set\ \ $x^- := 2\widetilde{x}-x_n$. 
Compute\ \ $\widehat{x} = P_{V}(x^-)$\ \ using \cref{example:orth}\cref{example:orth:vii} or \cref{example:box}\cref{example:box:vii}.
\item[Step 4] ({\em Update}) Set\ \ $x_{n+1} := x_n + 2\lambda(\widehat{x} - \widetilde{x})$.
\item[Step 5] ({\em Stopping criterion}) If $\|\widetilde{x} - \widetilde{x}_n\|_\infty \le \varepsilon$, then return $\widetilde{x}$, and $\gap$ using \cref{fact:pazybr}\cref{eqn:v1} and \cref{eqn:v2}, and stop.  
Otherwise, set $\widetilde{x}_{n+1}:=\widetilde{x}$, $n := n+1$, and go to Step 2.
\end{description}
\end{algorithm}

\noindent
\begin{algorithm}{\bf ({Dykstra})}\label{alg:dyk} \rm
\begin{description}
\item[Step 1] ({\em Initialization}) Choose the initial iterates $x_0=0,\,p_0=0$ and $q_0=0$. Choose the initial iterate $\widetilde{x}_0$ arbitrarily. Choose a small parameter $\varepsilon>0$, and set $n=0$. 
\item[Step 2] ({\em Projection onto $U$})  Set\ \ $x^- 
:= x_{n} + q_n$.  Compute\ \ $\widetilde{x} = P_U(x^-)$\ \ using \cref{example:orth}\cref{example:orth:vi} or \cref{example:box}\cref{example:box:vi}.
\item[Step 3] ({\em Projection onto $V$}) Set\ \ $x^- :=
\widetilde{x}+p_n$. Compute\ \ $\widehat{x} = P_V(x^-)$\ \ using \cref{example:orth}\cref{example:orth:vii} or \cref{example:box}\cref{example:box:vii}.
\item[Step 4] ({\em Update}) Set\ \ $x_{n+1} 
:= \widehat{x}$,\,\,
$q_{n+1} 
:= x_{n} + q_n - \widetilde{x}$\ \ and\ \ $p_{n+1} 
:= \widetilde{x} + p_n - \widehat{x}$\,.
\item[Step 5] ({\em Stopping criterion}) If $\|\widetilde{x} - \widetilde{x}_n\|_\infty \le \varepsilon$, then return $\widetilde{x}$, and $\gap$ using \cref{fact:pazybr}\cref{eqn:v1} and \cref{eqn:v2}, and stop.  
Otherwise, set $\widetilde{x}_{n+1}:=\widetilde{x}$, $n := n+1$, and go to Step 2.
\end{description}
\end{algorithm}

\subsection{Computations}

\cref{tab:orth}--\cref{tab:box} show the error in the objective $\phi$, the gap vector $\gap$ and the solution $\overline{x}$. The ``exact'' solution used in these calculations was found using Dykstra's algorithm with a tolerance $\epsilon=10^{-12}$ (except for \cref{example:box} when $(m,\,d) = (50,\,1000)$ with $\epsilon=3\times10^{-12}$ since Dykstra was unable to terminate with $\epsilon=10^{-12}$). For $d>1000$ we have not been able to \emph{consistently} reach our desired accuracy. Further investigation is needed to understand the cause of this issue.

In \Cref{fig:plots}, we show plots of the number of iterations taken by the DR relaxation with $\lambda\in \{0.25,\ 0.5,\ 0.75,\, 0.9,\ 1\}$, against the parameter $\gamma\in\left]0,1\right[$. We obtained these plots for $500$ values of $\gamma\in\left]0,1\right[$. For the sake of brevity we only show four of these plots but we have generated these for every problem instance found in \cref{tab:orth}--\cref{tab:box}. These plots have informed our choice of a value of $\gamma$ for each problem instance that provides optimal performance of the algorithms.

In \Cref{fig:plots} for \cref{example:box} with $(m,\,d) = (500,\,1000)$, $\lambda\in \{0.25,\,0.75\}$ and $(m,\,d) = (50,\,1000)$, we observe that for values of $\gamma$ close to 1, the methods terminate in less than 10 iterations though we have not arrived at the solution. In these examples we thus chose the next best values of $\gamma$. This also occurred for \cref{example:box} with $(m,\,d) = (10,\,100)$ so we again chose the next best value of $\gamma$. Though we have not included the figure for \cref{example:box} with $(m,\,d) = (65,\,70)$ this example showed $n=2$ for all $\gamma\in[0.6,1[$ and every value of $\lambda$, thus in \cref{tab:box} we set $\gamma=0.6$ but any $\gamma\in[0.6,1[$ could have been chosen.

\begin{figure}[H]
\begin{subfigure}{.5\textwidth}
    \centering
    \includegraphics[width=7.5cm]{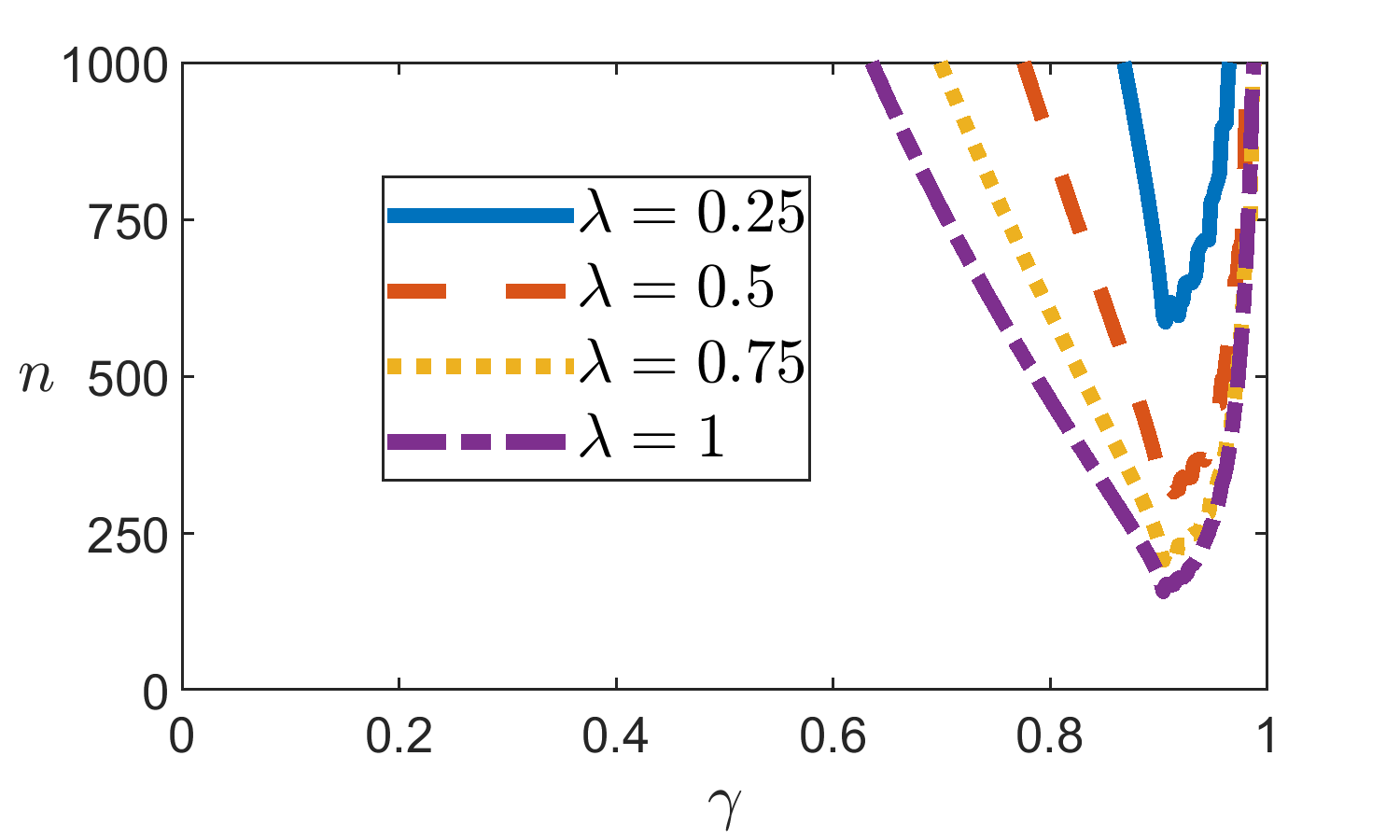}
    \includegraphics[width=7.5cm]{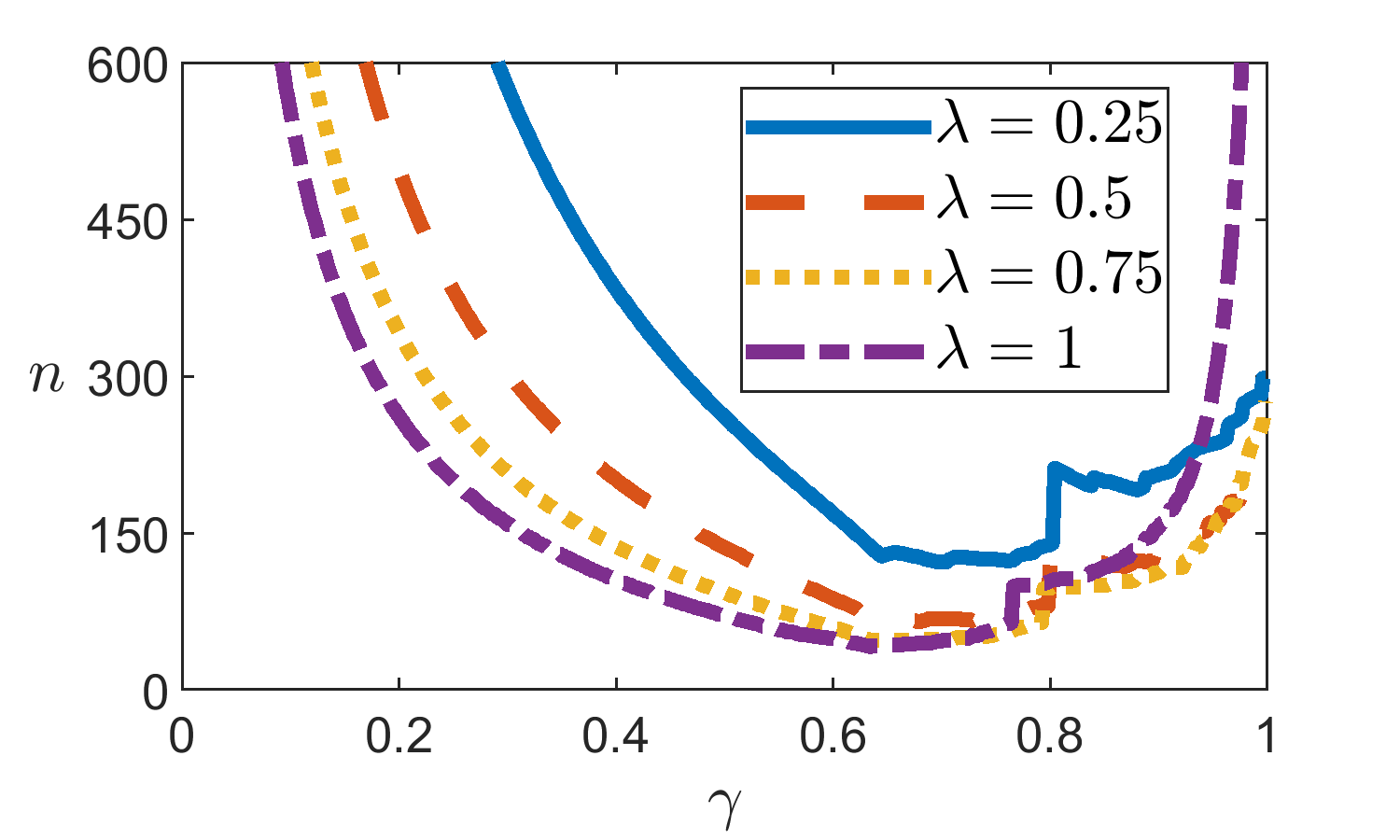}
    \caption{\cref{example:orth}.}
\end{subfigure}
\begin{subfigure}{.5\textwidth}
    \centering
    \includegraphics[width=7.5cm]{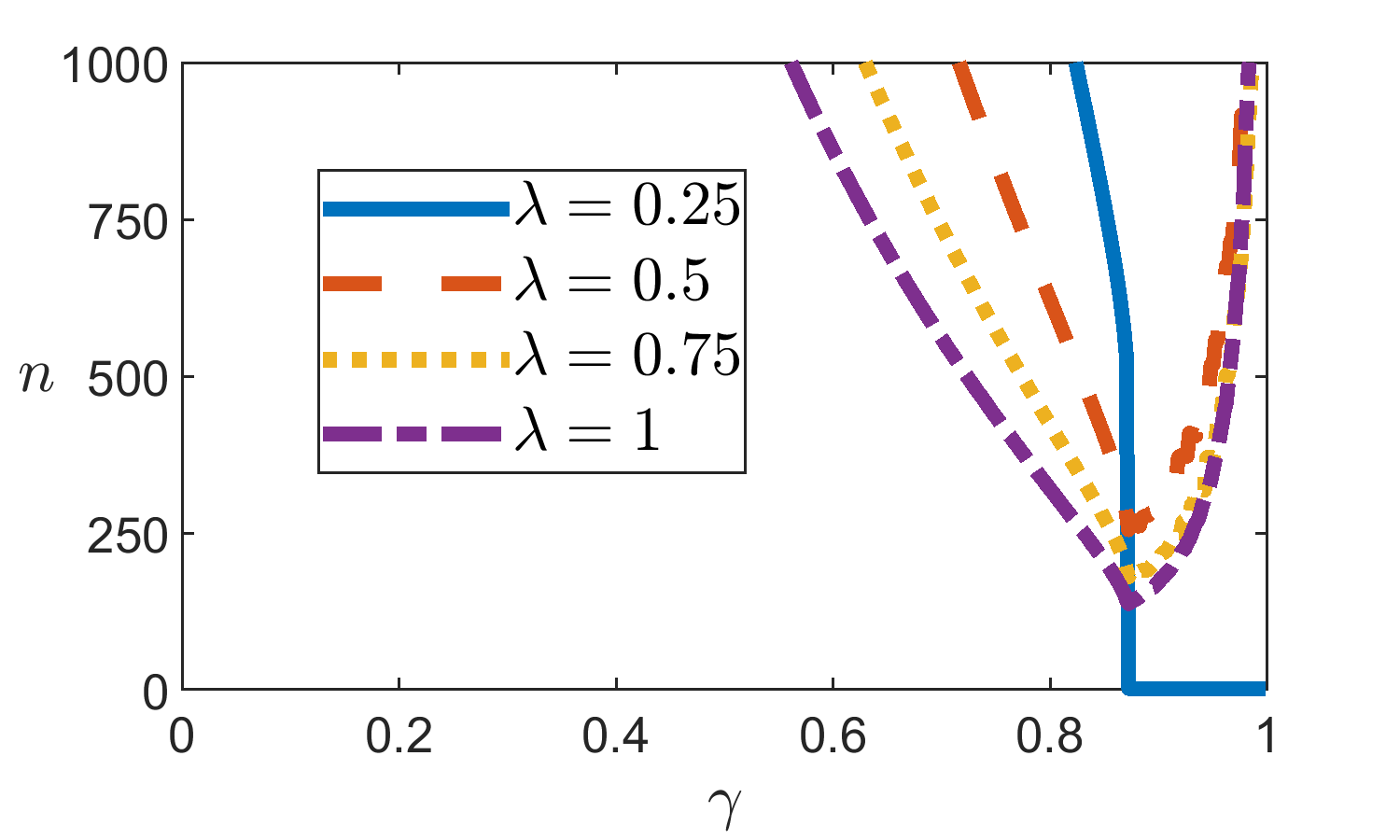}
    \includegraphics[width=7.5cm]{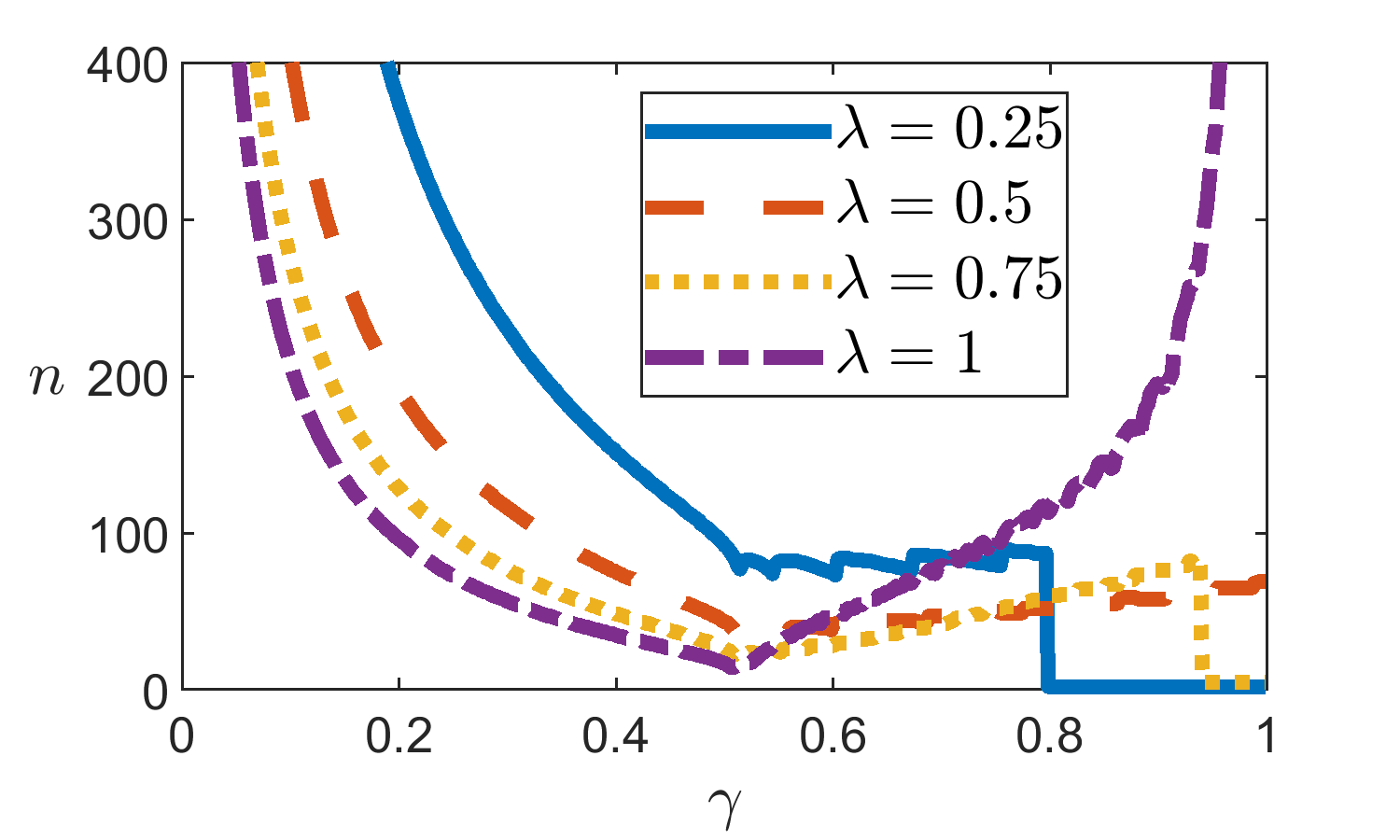}
    \caption{\cref{example:box}.}
\end{subfigure}
\caption{Parameter curves with $m=50,\,d=1000$ (top) and $m=500,\,d=1000$ (bottom).}
\label{fig:plots}
\end{figure}

From observing these figures we can make some comments on the choice of $\gamma$. One can see in all plots from \Cref{fig:plots} and the values of $\gamma$ used in \cref{tab:orth}--\cref{tab:box} that the optimal $\gamma$ has little variation as $\lambda$ changes compared to the variation between different problem instances. We have also generated plots for the same problem instances $(m,d)$ in \cref{tab:orth}--\cref{tab:box} but with different (randomly generated) matrices $L,b$ and have observed that for the majority the optimal $\gamma$ remains more or less unchanged for the same problem instance $(m,d)$. Thus it appears that the optimal value of $\gamma$ depends largely on the problem instance $(m,d)$ rather than the value of $\lambda$. Across all the problem instances we observe that the optimal values of $\gamma$ range from $0.5$ to $0.9$. In order to determine what specific value of $\gamma$ is optimal for a problem instance we need to make further comparisons which are carried out below.

When comparing the error in $\overline{x}$ and $\gap$ for different values of $\gamma$ we see that in general the values of $\gamma$ that take the smallest number of iterations also produce the smallest errors. This gives further merit to identifying the optimal value of $\gamma$ for each problem. In general, we observe that the more difficult a problem is to solve (i.e., the larger the number of iterations is), the larger the optimal value of $\gamma$ becomes. From \cref{tab:orth}--\cref{tab:box} we see in each example the case where $(m,\,d) = (65,\,70)$ takes the smallest number of iterations and has the smallest values of $\gamma$ while $(m,\,d) = (50,\,1000)$ takes the largest number of iterations and has the largest values of $\gamma$. Problem instances where $d$ is large and/or where $m\ll d$ seem to be more difficult for the algorithms to solve.

\begin{table}[t!]
    \centering
    \begin{tabular}{lllllllll}
        $m$ & $d$ & $|\phi(\overline{x})-\phi(\overline{x}^*)|$ & $\|\gap-\gap^*\|_{\infty}$ & $\|\gap-\gap^*\|_{\infty}$ & $\|\overline{x}-\overline{x}^*\|_{\infty}$ &  $\lambda$ & $\gamma$ & $n$ \\
        & & & (\cref{fact:pazybr}\ref{eqn:v1}) & (\cref{fact:pazybr}\ref{eqn:v2}) & & & & \\ \hline
        \multirow{6}{*}{$10$} & \multirow{6}{*}{$100$} & $1.22\times10^{-9}$ & $1.73\times10^{-3}$ & $2.48\times10^{-8}$ & $1.38\times10^{-7}$ & $0.25$ & $0.75$ & $204$ \\
        & & $6.65\times10^{-10}$ & $1.62\times10^{-3}$ & $1.06\times10^{-8}$ & $4.35\times10^{-8}$ & $0.5$ & $0.75$ & $109$ \\
        & & $1.03\times10^{-10}$ & $1.57\times10^{-3}$ & $4.82\times10^{-9}$ & $1.27\times10^{-8}$ & $0.75$ & $0.75$ & $75$ \\
        & & $2.61\times10^{-10}$ & $1.55\times10^{-3}$ & $2.18\times10^{-8}$ & $8.73\times10^{-9}$ & $0.9$ & $0.75$ & $63$ \\
        & & $5.17\times10^{-10}$ & $1.54\times10^{-3}$ & $-$ & $1.01\times10^{-8}$ & $1$ & $0.75$ & $57$ \\
        & & $2.82\times10^{-9}$ & $-$ & $1.74\times10^{-8}$ & $2.73\times10^{-7}$ & \multicolumn{2}{c}{Dykstra} & $338$ \\ \hline
        \multirow{6}{*}{$65$} & \multirow{6}{*}{$70$} & $8.77\times10^{-8}$ & $2.66\times10^{-2}$ & $1.54\times10^{-7}$ & $2.25\times10^{-7}$ & $0.25$ & $0.7$ & $44$ \\
        & & $4.94\times10^{-9}$ & $2.98\times10^{-2}$ & $6.95\times10^{-9}$ & $6.53\times10^{-9}$ & $0.5$ & $0.55$ & $25$ \\
        & & $3.09\times10^{-9}$ & $3.64\times10^{-2}$ & $2.46\times10^{-9}$ & $2.46\times10^{-9}$ & $0.75$ & $0.5$ & $15$ \\
        & & $8.03\times10^{-10}$ & $2.59\times10^{-2}$ & $1.94\times10^{-9}$ & $1.53\times10^{-9}$ & $0.9$ & $0.55$ & $16$ \\
        & & $8.96\times10^{-10}$ & $1.62\times10^{-2}$ & $-$ & $1.95\times10^{-9}$ & $1$ & $0.55$ & $23$ \\
        & & $1.35\times10^{-9}$ & $-$ & $2.20\times10^{-9}$ & $4.15\times10^{-9}$ & \multicolumn{2}{c}{Dykstra} & $22$ \\ \hline
        \multirow{6}{*}{$50$} & \multirow{6}{*}{$1000$} & $2.81\times10^{-8}$ & $3.05\times10^{-4}$ & $2.15\times10^{-8}$ & $4.24\times10^{-7}$ & $0.25$ & $0.9$ & $662$ \\
        & & $1.25\times10^{-8}$ & $2.93\times10^{-4}$ & $1.04\times10^{-8}$ & $2.07\times10^{-7}$ & $0.5$ & $0.9$ & $345$ \\
        & & $7.31\times10^{-9}$ & $2.85\times10^{-4}$ & $6.49\times10^{-9}$ & $1.28\times10^{-7}$ & $0.75$ & $0.9$ & $236$ \\
        & & $5.76\times10^{-9}$ & $2.82\times10^{-4}$ & $5.30\times10^{-9}$ & $1.02\times10^{-7}$ & $0.9$ & $0.9$ & $199$ \\
        & & $4.94\times10^{-9}$ & $2.79\times10^{-4}$ & $-$ & $8.64\times10^{-8}$ & $1$ & $0.9$ & $181$ \\
        & & $2.89\times10^{-7}$ & $-$ & $3.96\times10^{-8}$ & $3.28\times10^{-6}$ & \multicolumn{2}{c}{Dykstra} & $3078$ \\ \hline
        \multirow{6}{*}{$500$} & \multirow{6}{*}{$1000$} & $3.74\times10^{-8}$ & $2.48\times10^{-3}$ & $2.37\times10^{-8}$ & $9.42\times10^{-8}$ & $0.25$ & $0.7$ & $122$ \\
        & & $3.56\times10^{-9}$ & $2.22\times10^{-3}$ & $9.70\times10^{-9}$ & $2.13\times10^{-8}$ & $0.5$ & $0.7$ & $68$ \\
        & & $3.08\times10^{-9}$ & $2.31\times10^{-3}$ & $6.73\times10^{-9}$ & $1.73\times10^{-8}$ & $0.75$ & $0.65$ & $47$ \\
        & & $1.25\times10^{-9}$ & $2.20\times10^{-3}$ & $1.50\times10^{-6}$ & $1.16\times10^{-8}$ & $0.9$ & $0.65$ & $41$ \\
        & & $1.56\times10^{-10}$ & $1.89\times10^{-3}$ & $-$ & $2.70\times10^{-9}$ & $1$ & $0.65$ & $43$ \\
        & & $3.72\times10^{-8}$ & $-$ & $1.37\times10^{-8}$ & $8.44\times10^{-8}$ & \multicolumn{2}{c}{Dykstra} & $135$ \\
        \end{tabular}
    \caption{\cref{example:orth} with $z=\mathbf{0}$, $\epsilon=10^{-8}$, $x\in\mathbb{R}^d_+$. For Dykstra, \eqref{eqn:v_Dyk} is used to calculate $\gap$.}
    \label{tab:orth}
\end{table}

From \cref{tab:orth}--\cref{tab:box} the relaxed DR algorithm with $\lambda=0.75$ appears to provide the best performance in general. We see a trend that as $\lambda$ increases the number of iterations taken decreases. Though PR ($\lambda=1$) needed the smallest number of iterations to reach a solution of the desired tolerance we cannot apply \cref{fact:pazybr}\cref{eqn:v2} to calculate $\gap$ and instead must rely on \cref{fact:pazybr}\cref{eqn:v1}. From \cref{tab:orth}--\cref{tab:box} we observe that the errors in $\gap$ when using \cref{fact:pazybr}\cref{eqn:v2} are almost always significantly smaller than those using \cref{fact:pazybr}\cref{eqn:v1}. The limit expression in \cref{fact:pazybr}\cref{eqn:v1} contains $n$ thus to more accurately compute $\gap$ we must allow the algorithm to run for more iterations which defeats the performance benefit gained by choosing this method. Thus one could conclude that to find a solution quickly and accurately the relaxed DR algorithm is preferred over PR. When $\lambda=0.9$ \cref{fact:pazybr}\cref{eqn:v2} is applicable but we see from \cref{tab:orth}--\cref{tab:box} that the errors in $\gap$ are often much larger than those from $\lambda\in\{0.25,\, 0.5,\, 0.75\}$ or from Dykstra.

\begin{table}[t!]
    \centering
    \begin{tabular}{lllllllll}
        $m$ & $d$ & $|\phi(\overline{x})-\phi(\overline{x}^*)|$ & $\|\gap-\gap^*\|_{\infty}$ & $\|\gap-\gap^*\|_{\infty}$ & $\|\overline{x}-\overline{x}^*\|_{\infty}$ &  $\lambda$ & $\gamma$ & $n$ \\
        & & & (\cref{fact:pazybr}\ref{eqn:v1}) & (\cref{fact:pazybr}\ref{eqn:v2}) & & & & \\ \hline
        \multirow{6}{*}{$10$} & \multirow{6}{*}{$100$} & $9.44\times10^{-8}$ & $1.04\times10^{-1}$ & $2.62\times10^{-8}$ & $5.16\times10^{-8}$ & $0.25$ & $0.65$ & $151$ \\
        & & $1.32\times10^{-8}$ & $1.01\times10^{-1}$ & $8.03\times10^{-9}$ & $7.22\times10^{-9}$ & $0.5$ & $0.65$ & $73$ \\
        & & $3.19\times10^{-8}$ & $1.13\times10^{-1}$ & $1.49\times10^{-9}$ & $1.74\times10^{-8}$ & $0.75$ & $0.65$ & $46$ \\
        & & $6.80\times10^{-8}$ & $1.21\times10^{-1}$ & $5.99\times10^{-4}$ & $3.71\times10^{-8}$ & $0.9$ & $0.65$ & $36$ \\
        & & $9.45\times10^{-8}$ & $1.26\times10^{-1}$ & $-$ & $5.16\times10^{-8}$ & $1$ & $0.65$ & $31$ \\
        & & $1.91\times10^{-7}$ & $-$ & $8.70\times10^{-9}$ & $1.04\times10^{-7}$ & \multicolumn{2}{c}{Dykstra} & $183$ \\ \hline
        \multirow{6}{*}{$65$} & \multirow{6}{*}{$70$} & $5.68\times10^{-14}$ & $4.70\times10^{-1}$ & $3.14\times10^{-1}$ & $0$ & $0.25$ & $0.6$ & $2$ \\
        & & $5.68\times10^{-14}$ & $3.14\times10^{-1}$ & $3.42\times10^{-14}$ & $0$ & $0.5$ & $0.6$ & $2$ \\
        & & $5.68\times10^{-14}$ & $1.57\times10^{-1}$ & $3.14\times10^{-1}$ & $0$ & $0.75$ & $0.6$ & $2$ \\
        & & $5.68\times10^{-14}$ & $6.27\times10^{-2}$ & $5.02\times10^{-1}$ & $0$ & $0.9$ & $0.6$ & $2$ \\
        & & $5.68\times10^{-14}$ & $2.86\times10^{-14}$ & $-$ & $0$ & $1$ & $0.6$ & $2$ \\
        & & $5.68\times10^{-14}$ & $-$ & $1.09\times10^{-13}$ & $0$ & \multicolumn{2}{c}{Dykstra} & $3$ \\ \hline
        \multirow{6}{*}{$50$} & \multirow{6}{*}{$1000$} & $2.14\times10^{-6}$ & $2.77\times10^{-2}$ & $2.29\times10^{-8}$ & $4.76\times10^{-7}$ & $0.25$ & $0.85$ & $772$ \\
        & & $2.71\times10^{-7}$ & $3.50\times10^{-2}$ & $1.00\times10^{-8}$ & $1.04\times10^{-7}$ & $0.5$ & $0.9$ & $305$ \\
        & & $1.91\times10^{-7}$ & $3.48\times10^{-2}$ & $4.90\times10^{-9}$ & $7.42\times10^{-8}$ & $0.75$ & $0.9$ & $204$ \\
        & & $9.86\times10^{-8}$ & $3.24\times10^{-2}$ & $4.41\times10^{-9}$ & $2.95\times10^{-8}$ & $0.9$ & $0.9$ & $182$ \\
        & & $7.38\times10^{-8}$ & $3.23\times10^{-2}$ & $-$ & $3.58\times10^{-8}$ & $1$ & $0.9$ & $172$ \\
        & & $7.82\times10^{-6}$ & $-$ & $3.17\times10^{-8}$ & $1.74\times10^{-6}$ & \multicolumn{2}{c}{Dykstra} & $2403$ \\ \hline
        \multirow{6}{*}{$500$} & \multirow{6}{*}{$1000$} & $2.75\times10^{-6}$ & $8.29\times10^{-2}$ & $3.85\times10^{-8}$ & $1.42\times10^{-7}$ & $0.25$ & $0.6$ & $74$ \\
        & & $1.19\times10^{-7}$ & $1.09\times10^{-1}$ & $2.99\times10^{-9}$ & $7.30\times10^{-9}$ & $0.5$ & $0.55$ & $38$ \\
        & & $8.01\times10^{-8}$ & $1.50\times10^{-1}$ & $1.17\times10^{-7}$ & $5.16\times10^{-9}$ & $0.75$ & $0.5$ & $26$ \\
        & & $3.62\times10^{-8}$ & $1.76\times10^{-1}$ & $2.96\times10^{-2}$ & $2.39\times10^{-9}$ & $0.9$ & $0.5$ & $20$ \\
        & & $1.38\times10^{-8}$ & $2.46\times10^{-1}$ & $-$ & $9.27\times10^{-10}$ & $1$ & $0.5$ & $17$ \\
        & & $1.87\times10^{-7}$ & $-$ & $4.63\times10^{-9}$ & $1.17\times10^{-8}$ & \multicolumn{2}{c}{Dykstra} & $42$ \\
    \end{tabular}
    \caption{\cref{example:box} with $z=5\times\mathbf{1}$, $\epsilon=10^{-8}$, $2\leq x_i\leq10$,  $i\in \{1,\dots,d\}$. For Dykstra, \eqref{eqn:v_Dyk} is used to calculate $\gap$.}
    \label{tab:box}
\end{table}

To summarize our observations from the numerical experiments:
\begin{itemize}
    \item When considering our measures of optimal performance, i.e. a small number of iterations and an accurate approximation of $\overline{x}$ and $\gap$, the relaxed DR algorithm with $\lambda=0.75$ seems to be the best choice.
    \item When $d$ is large and/or where $m\ll d$ an optimal $\gamma$ is likely to be $\gamma\in[0.8,1[$, else $\gamma\in[0.5,0.8]$ may be an optimal choice.
\end{itemize}

\section{Conclusion and Future Work}
\label{sec:conc}
We have proved the strong convergence of the shadow sequences of the Douglas--Rachford algorithm and Peaceman--Rachford algorithm in the inconsistent case. We have conducted numerical experiments comparing the performance of the relaxed Douglas--Rachford algorithm and Peaceman--Rachford algorithm on inconsistent finite-dimensional optimization problems of various sizes.

In the future it will be interesting to see the performance of these algorithms on inconsistent infinite-dimensional optimization problems such as the optimal control problems studied in \cite{BurKayMou2024}:
\[
\mbox{(MEC)\ }\left\{\begin{array}{rl}
\ds\min_{u\in L^2([0,1],\RR^m)} &\ \ds\int_0^1\|u(t)\|_2^2\,dt  \\[4mm]
s.t. &\ \dot{x}(t) = L(t)\,x(t) + B(t)\,u(t)\,,\ \ x(0) = x_0,\ \ x(1) = x_f\,, \\[2mm]
    &\ \underline{u}_i \le u_i(t) \le \overline{u}_i\,,\ \ \mbox{for all}\  t\in[0,1]\,,\ \ i = 1,\ldots,m\,,
\end{array} \right.
\]
where $x$ and $u$ are the {\em state} and {\em control variables}, respectively.
Problem~(MEC) is referred to as the minimum-energy control problem which is a special case of linear--quadratic optimal control problems.  This problem is more general and computationally more challenging than the finite-dimensional linear--quadratic problems we have made experiments with in this paper.  

We note that the consistent case of Problem~(MEC) has been extensively studied in~\cite{BauBurKay2019, BurCalKay2024a, BurCalKayMou2024}, including its extension to an {\em LQ control problem} subject to both control and state box constraints~\cite{BurCalKay2024b}, all for $\lambda = 1/2$. One must also note that \cref{thm:DR:convergence:fg}--\cref{thm:PR:convergence:fg} hold also for the consistent case.  So it will be interesting to solve the optimal control problems in \cite{BurCalKay2024a, BurCalKayMou2024,BurCalKay2024b}, for various $\lambda\in]0,1]$ and make comparisons, in the consistent case.

\section*{Acknowledgements}
The research of WMM is partially supported by 
the Natural Sciences and Engineering Research Council of
Canada Discovery Grant (NSERC DG).
The research of BIC was supported by an Australian Government Research Training Program Scholarship.
The research of MS is partially supported by WMM's NSERC DG.

\end{document}